\documentclass[a4paper, 12pt]{amsart}
\usepackage[type=paper, externalize=false, font=libertine, footnotes, math=fancy]{styles/style}
\usepackage[margin=3cm]{geometry}
\pdfoutput=1

\usepackage{rotating}
\usepackage{floatpag}
\newenvironment{amssidewaysfigure}{%
  \begin{sidewaysfigure}%
    \vspace*{.5\textwidth}%
    \begin{minipage}{\textheight}%
      \thisfloatpagestyle{empty}%
      \centering%
}{%
    \end{minipage}%
  \end{sidewaysfigure}%
}

\newcommand*{\bbz}[1][B]{\ensuremath{{#1}^{\bullet}_0}\xspace}
\newcommand*{\bcz}[1][B]{\ensuremath{{#1}^{\circ}_0}\xspace}
\newcommand*{\bbt}[1][B]{\ensuremath{{#1}^{\bullet}_2}\xspace}
\newcommand*{\bct}[1][B]{\ensuremath{{#1}^{\circ}_2}\xspace}
\newcommand*{\duoidal}[1][\cat{D}]{\ensuremath{(#1, \circ, \bot, \bullet, 1)}\xspace}
\newcommand*{\duoidalb}[1][\cat{D}]{\ensuremath{(#1, \bullet, 1)}\xspace}
\newcommand*{\duoidalw}[1][\cat{D}]{\ensuremath{(#1, \circ, \bot)}\xspace}

\makeatletter
\newtheorem*{rep@theorem}{\rep@title}
\newenvironment{reptheorem}[1]{%
  \def\rep@title{\cref{#1}}%
  \begin{rep@theorem}%
}{%
  \end{rep@theorem}%
}
\makeatother

\begin{document}

\author{Tony Zorman}
\address{TU Dresden, Institut für Geometrie, Zellescher Weg 12--14, 01062 Dresden, Germany}
\email{tony.zorman@tu-dresden.de}

\title{Duoidal R-Matrices}
\maketitle

\begin{abstract}
  In this note, we define an analogue of R-matrices for bialgebras in the setting of a monad that is opmonoidal over two tensor products.
  Analogous to the classical case, such structures bijectively correspond to duoidal structures on the Eilenberg–Moore category of the monad.
  Further, we investigate how a cocommutative version of this lifts the linearly distributive structure of a normal duoidal category.
\end{abstract}

\TOC{}

\section{Introduction}

\noindent Monadic \emph{reconstruction} theory%
---relating additional structure on a monad to structure on its category of algebras---%
has a long tradition.
For example, such results were proved for bimonads\footnote{
  \,Bimonads are called ``Hopf monads'' in~\cite{Moerdijk2002};
  we follow the nomenclature of~\cite{Bruguieres2007,Bruguieres2011} and reserve that term for monads lifting the rigid or closed structure of their base category.
} in~\cite{Moerdijk2002,McCrudden2002},
for Hopf monads in~\cite{Bruguieres2007},
for comodule monads in~\cite{Aguiar2012,halbig24:pivot-hopf},
and for \(*\)-autonomous and linearly distributive monads in~\cite{pastro09:closed,pastro12:note}.

\emph{Duoidal categories} were introduced in~\cite{Aguiar2010} under the name of \emph{2-monoidal categories},
generalising braided monoidal categories
by considering two monoidal structures that are connected by a non-invertible interchange law.
They also generalise 2-fold monoidal categories in the sense of~\cite{balteanu03:iterat},
where the two tensor products are assumed to share a unit.
The terminology used here is due to~\cite[Definition~3]{batanin12:center}.
These structures have been used to study higher-dimensional Hopf theory%
~\cite{boehm13:hopf,booker13:tannak,aguiar18:monad,franco20:duoid,boehm21:masch-hopf},
and have also found applications in
various other fields of mathematics;
see for example~\cite{garner16:commut,shapiro22:duoid-struc-compos-depen,román24:strin-diagr-physic-duoid-categ,torii24:map}.

This note generalises a reconstruction-type result for R-matrices on bimonads,%
~\cite[Proposition~8.5]{Bruguieres2007},
which in turn generalises the classical theory of R-matrices for bialgebras.
The former has the additional advantage of not requiring a braided monoidal base category, as bimonads%
---in contrast with bialgebras---%
may be defined on any monoidal category.

As such, we introduce the notion of an R-matrix for a monad \(T\) on a preduoidal category%
—one equipped with two monoidal structures—%
that is opmonoidal with respect to each one individually.

In \cref{sec:preliminaries} we introduce notation and discuss preliminary results on duoidal categories and bimonads.
\Cref{sec:r-matrices} first discusses a generalisation of cocommutative bialgebras
in the form of the double opmonoidal monads of~\cite[Section~7]{aguiar18:monad},
and then generalises this notion to the non-cocommutative setting
by introducing R-matrices over separately opmonoidal monads on preduoidal categories.
Our main result is:
\begin{reptheorem}{thm:r-matrices-iff-duoidal-structure}
  Let \(\cat{D}\) be a category with monoidal structures \(\circ\) and \(\bullet\),
  and \(T\) a monad on \(\cat{D}\) that has a \(\circ\)-opmonoidal and a \(\bullet\)-opmonoidal structure.
  Then quasi\-triangular structures on \(T\) are in bijective correspondence with duoidal structures on \(\cat{D}^T\).
\end{reptheorem}

\Cref{sec:lin-dist-bimonads} studies the relationship between normal duoidal
and linearly distributive categories from the monadic point of view.
In particular, we see non-planar linearly distributive categories \(\cat{L}\) as an analogue of preduoidal categories,
in the sense that the additional structure trivialises in the monoidal case,
see~\cref{ex:linear-dist-bimonad-in-monoidal-setting}.
Equipping \(\cat{L}\) with a planar structure,
we can relate double comonoidal monads
to linearly distributive monads in the sense of~\cite{pastro12:note}.

\numberwithin{theorem}{section}
\addtocontents{toc}{\SkipTocEntry}
\subsection*{Acknowledgements}\label{sec:acknowledgements}
We thank Ulrich Krähmer for many useful comments and suggestions on a first draft of this article,
and Marcelo Aguiar for clarifying parts of~\cite{aguiar18:monad}.
The author is supported by \textsc{dfg} grant \textsc{kr} \oldstylenums{5036}/\oldstylenums{2}--\oldstylenums{1}.

\section{Preliminaries}\label{sec:preliminaries}

\noindent We refer the reader to~\cite{MacLane1998} and~\cite{Etingof2015} for comprehensive textbook accounts on category theory and monoidal categories.

\begin{definition}\label{def:bimonad}
  A monad \((B, \mu, \eta)\) on a monoidal category \(\cat{C}\) is called a \emph{bimonad}
  if \(B\) is an opmonoidal functor for which \(\mu\) and \(\eta\) are opmonoidal natural transformations.
\end{definition}

For a monad \(T\) on a category \(\cat{C}\),
we denote the \emph{Eilenberg–Moore category} of \(T\),
also called the category of \emph{\(T\)-algebras} or \emph{\(T\)-modules},
by \(\cat{C}^T\).
The following reconstruction result was observed by
Moerdijk~\cite[Theorem~7.1]{Moerdijk2002}
and McCrudden~\cite[Corollary~3.13]{McCrudden2002}.

\begin{proposition}\label{prop: Moerdijk_reconstruction}
  Let \((B, \mu, \eta)\) be a monad on a monoidal category \(\cat{C}\).
  There exists a bijective correspondence between
  bimonad structures on \(B\) and
  monoidal structures on \(\cat{C}^B\)
  such that the canonical forgetful functor \(U^B\from \cat{C}^{B} \to \cat{C}\) is strict monoidal.
\end{proposition}

The next definition first appeared in~\cite[Definition~6.1]{Aguiar2010}
under the name \emph{2-monoidal category}.
We follow the nomenclature of~\cite[Definition~3]{batanin12:center}
and the notation of~\cite{boehm13:hopf}.

\begin{definition}\label{def:duoidal-category}
  A \emph{duoidal} category is a quintuple \duoidal, consisting of
  \begin{itemize}[label=\raisebox{0.25ex}{\tiny\(\bullet\)}]
    \item monoidal categories \duoidalb{} and \duoidalw;
    \item a not-necessarily invertible natural transformation
    \[
      \zeta_{x,y,a,b} \colon (x \bullet y) \circ (a \bullet b) \nt (x \circ a) \bullet (y \circ b),
    \]
    called the \emph{middle interchange law};
    \item three \emph{structure morphisms}
    \[
      \nu \colon \bot \to \bot \bullet \bot, \qquad \varpi \colon 1 \circ 1 \to 1, \qquad \iota \colon \bot \to 1;
    \]
  \end{itemize}
  such that:
  \begin{itemize}[label=\raisebox{0.25ex}{\tiny\(\bullet\)}]
    \item \((1, \varpi, \iota)\) is a monoid in \duoidalw;
    \item \((\bot, \nu, \iota)\) is a comonoid in \duoidalb;
    \item the following diagrams commute, witnessing \emph{associativity}:
    \begin{equation} \label[diagram]{eq:middle-interchange-assoc1}
      \begin{tikzcd}[ampersand replacement=\&]
	{((x \bullet y) \circ (a \bullet b)) \circ (c \bullet d)} \& {(x \bullet y) \circ ((a \bullet b) \circ (c \bullet d))} \\
	{((x \circ a) \bullet (y \circ b)) \circ (c \bullet d)} \& {(x \bullet y) \circ ((a \circ c) \bullet (b \circ d))} \\
	{((x \circ a) \circ c) \bullet ((y \circ b) \circ d)} \& {(x \circ (a \circ c)) \bullet (y \circ (b \circ d))}
	\arrow["\alpha", from=1-1, to=1-2]
	\arrow["{\id \circ \zeta}", from=1-2, to=2-2]
	\arrow["{\zeta}", from=2-2, to=3-2]
	\arrow["{\zeta \circ \id}"', from=1-1, to=2-1]
	\arrow["{\zeta}"', from=2-1, to=3-1]
	\arrow["{\alpha \bullet \alpha}"', from=3-1, to=3-2]
      \end{tikzcd}
    \end{equation}
    \begin{equation} \label[diagram]{eq:middle-interchange-assoc2}
      \begin{tikzcd}[ampersand replacement=\&]
	{((x \bullet a) \bullet c) \circ ((y \bullet b) \bullet d)} \& {(x \bullet (a \bullet c)) \circ (y \bullet (b \bullet d))} \\
	{((x \bullet a) \circ (y \bullet b)) \bullet (c \circ d)} \& {(x \circ y) \bullet ((a \bullet c) \circ (b \bullet d))} \\
	{((x \circ y) \bullet (a \circ b)) \bullet (c \circ d)} \& {(x \circ y) \bullet ((a \circ b) \bullet (c \circ d))}
	\arrow["{\alpha \circ \alpha}", from=1-1, to=1-2]
	\arrow["{\zeta}", from=1-2, to=2-2]
	\arrow["{\id \bullet \zeta}", from=2-2, to=3-2]
	\arrow["{\zeta}"', from=1-1, to=2-1]
	\arrow["{\zeta \bullet \id}"', from=2-1, to=3-1]
	\arrow["\alpha"', from=3-1, to=3-2]
      \end{tikzcd}
    \end{equation}
    \item the following diagrams commute, witnessing \emph{unitality}:
    \begin{equation}\label[diagram]{eq:duoidal-cat-unitality}
      \scalebox{0.9}{\begin{tikzcd}[ampersand replacement=\&]
        {\bot \circ (a \bullet b)} \& {(\bot \bullet \bot) \circ (a \bullet b)} \& {(a \bullet b) \circ \bot} \& {(a \bullet b) \circ ( \bot \bullet \bot)} \\
        {a \bullet b} \& {(\bot \circ a) \bullet (\bot \circ b)} \& {a \bullet b} \& {(a \circ \bot) \bullet (b \circ \bot)} \\
        {(1 \bullet a) \circ (1 \bullet b)} \& {(1 \circ 1) \bullet (a \circ b)} \& {(a \bullet  1) \circ (b \bullet 1)} \& {(a \circ b) \bullet (1 \circ 1)} \\
        {a \circ b} \& {1 \bullet (a \circ b)} \& {a \circ b} \& {(a \circ b) \bullet 1}
        \arrow["\lambda"', from=4-1, to=4-2]
        \arrow["{\lambda \circ \lambda}"', from=3-1, to=4-1]
        \arrow["{\zeta}", from=3-1, to=3-2]
        \arrow["{\varpi \bullet \id}", from=3-2, to=4-2]
        \arrow["\lambda"', from=4-3, to=4-4]
        \arrow["{\lambda \circ \lambda}"', from=3-3, to=4-3]
        \arrow["{\zeta}", from=3-3, to=3-4]
        \arrow["{\id \bullet \varpi}", from=3-4, to=4-4]
        \arrow["{\nu \circ \id}", from=1-1, to=1-2]
        \arrow["{\zeta}", from=1-2, to=2-2]
        \arrow["\lambda"', from=1-1, to=2-1]
        \arrow["{\lambda^{-1} \bullet \lambda^{-1}}"', from=2-1, to=2-2]
        \arrow["{\id \circ \nu}", from=1-3, to=1-4]
        \arrow["{\zeta}", from=1-4, to=2-4]
        \arrow["\lambda"', from=1-3, to=2-3]
        \arrow["{\lambda^{-1} \bullet \lambda^{-1}}"', from=2-3, to=2-4]
      \end{tikzcd}}
    \end{equation}
  \end{itemize}
\end{definition}

By abuse of notation, we shall often call \(\cat{D}\) a duoidal category,
leaving the rest of the data implicit.

\begin{definition}\label{def:normal-duoidal-category}
  A duoidal category \(\cat{D}\) is called \emph{normal} if \(\bot \cong 1\).
\end{definition}

Note that explicitly requiring the existence of \(\iota \colon \bot \to 1\)
in \Cref{def:duoidal-category} is not strictly necessary,
as it may be derived from the other specified data:
\[
  \iota \from \bot
  \xrightarrow{\;\lambda\;} \bot \circ \bot
  \xrightarrow{\lambda \circ \rho} (1 \bullet \bot) \circ (\bot \bullet 1)
  \xrightarrow{\ \zeta\ } (1 \circ \bot) \bullet (\bot \circ 1)
  \xrightarrow{\lambda \bullet \rho} 1 \bullet 1
  \xrightarrow{\;\lambda\;} 1.
\]

\begin{example}\label{ex:braided-cat-is-duoidal}
  Let \((\cat{C}, \otimes, 1)\) be a braided monoidal category with braiding \(\sigma\).
  By~\cite[Proposition~6.10]{Aguiar2010},
  \((\cat{C}, \otimes, 1, \otimes, 1)\) becomes a duoidal category with structure morphisms
  \[
    \zeta
    \defeq (a \otimes b) \otimes (c \otimes d)
    \cong a \otimes (b \otimes c) \otimes d
    \xrightarrow{a \otimes \sigma_{b,c} \otimes d} a \otimes (c \otimes b) \otimes d
    \cong (a \otimes c) \otimes (b \otimes d),
  \]
  \[
    \varpi \defeq 1 \otimes 1 \xrightarrow{\ \lambda\ } 1, \qquad\qquad
    \nu \defeq 1 \xrightarrow{\ \lambda^{-1}\ } 1 \otimes 1, \qquad\qquad
    \iota \defeq 1 \xrightarrow{\ \id_1\ } 1.\footnote{\,Note in particular that \(\rho_1 = \lambda_1\).}
  \]
\end{example}

\begin{example}\label{ex:strong-duoidal-is-braided}
  The converse of \cref{ex:braided-cat-is-duoidal} also holds;
  if \(\cat{D}\) is a duoidal category,
  such that the interchange law and structure morphisms are isomorphisms,
  then~\cite[Proposition~6.11]{Aguiar2010}
  yields a braiding on \(\duoidalw\) and \(\duoidalb\),
  such that they become isomorphic as braided monoidal categories,
  and the interchange law arises from the braiding.

  Note, however, that there exist non-trivial duoidal structures on a monoidal category \((\cat{C}, \otimes, 1)\).
  For example, the category of (left-left) Yetter–Drinfeld modules over a Hopf algebra \(H\) with non-invertible antipode is \emph{lax braided}%
  —the Yetter–Drinfeld braiding
  \[
    \sigma \defeq {\{\, \sigma_{M,N} \from M \otimes N \to N \otimes M, \quad m \otimes n \mapsto m_{(-1)}n \otimes m_{(0)} \,\}}_{M,N \,\in\, \YDM{H}}
  \]
  is a non-invertible natural transformation that satisfies the braid equations.
  This yields a duoidal structure on \((\cat{C}, \otimes, 1, \otimes, 1)\) that is not braided.
\end{example}

There are various equivalent definitions of duoidal categories.
For example, as pseudomonoids in the monoidal 2-category of monoidal categories,
oplax monoidal functors, and oplax monoidal natural transformations~\cite[Definition~1]{garner16:commut}.
In particular, this means that \(\bullet\) is a lax monoidal and \(\circ\) is an oplax monoidal functor;
\footnote{%
  \,This extends to normal duoidal categories,
  in which \(\circ\from \cat{D} \times \cat{D} \to \cat{D}\) and \(\bot\from \1 \to \cat{D}\)
  are normal oplax monoidal functors,
  where \(\1\) is the terminal category.%
}
from this characterisation, one may obtain a \emph{coherence} result for these structures.

\begin{proposition}[{\cite{lewis72:coher},~\cite[Section~6.2]{Aguiar2010}}]\label{prop:duoidal-coherence}
  Any \textsc{etc} diagram in a duoidal category commutes.
\end{proposition}

Loosely speaking,
an \textsc{etc} diagram is a \emph{formal} diagram \(F\from \cat{J} \to \cat{D}\)
in the sense of~\cite[p.~20]{malkiewich22:coher},
comprising of only structure morphisms of the duoidal category,
such that for all \(j \in \cat{J}\)
the object \(Fj\) is non-isomorphic to any of the two units.
We refer to~\cite[Definition~5.8 and Theorem~5.9]{malkiewich22:coher} for a precise definition and a proof of the result.
A counterexample in the case of a formal diagram with parallel arrows \(1 \bullet 1 \rightrightarrows 1\) is given in~\cite[Proposition~3.1.6 and Example~3.1.7]{roman23:monoid-con}.

Note that%
—the tensor product and unit being normal monoidal functors—%
normal duoidal categories admit an analogue of the well-known coherence result for braided monoidal categories:
any formal diagram comprised only of the structure morphisms in a normal duoidal category commutes, see~\cite[Theorem~5.18]{malkiewich22:coher}.

\section{Quasitriangularity}\label{sec:r-matrices}

\subsection{Double opmonoidal monads}\label{sec:duoidal-bimonads}

\begin{definition}[{\cite[Definition~6.25]{Aguiar2010}}]\label{def:duoidal-bimonoid}
  Suppose that \(\cat{D}\) is a duoidal category.
  A \emph{bimonoid} in \(\cat{D}\) is a quintuple \((B, \mu, \eta, \Delta, \varepsilon)\),
  consisting of a monoid \((B, \mu, \eta)\) in \duoidalw,
  and a comonoid \((B, \Delta, \varepsilon)\) in \duoidalb,
  such that the following diagrams commute:
  \[
    \begin{tikzcd}[ampersand replacement=\&]
      {B \circ B} \& B \& {B \bullet B} \\
      {(B \bullet B) \circ (B \bullet B)} \&\& {(B \circ B) \bullet (B \circ B)}
      \arrow["{\Delta \circ \Delta}"', from=1-1, to=2-1]
      \arrow["{\zeta}"', from=2-1, to=2-3]
      \arrow["{\mu \bullet \mu}"', from=2-3, to=1-3]
      \arrow["\mu", from=1-1, to=1-2]
      \arrow["\Delta", from=1-2, to=1-3]
    \end{tikzcd}
  \]
  \[
    \begin{tikzcd}[ampersand replacement=\&]
      {B \circ B} \& {1 \circ 1} \& \bot \& B \& \bot \\
      B \& 1 \& {\bot \bullet \bot} \& {B \bullet B} \& B \& 1
      \arrow["{\varepsilon \circ \varepsilon}", from=1-1, to=1-2]
      \arrow["{\varpi}", from=1-2, to=2-2]
      \arrow["\mu"', from=1-1, to=2-1]
      \arrow["\varepsilon"', from=2-1, to=2-2]
      \arrow["\eta", from=1-3, to=1-4]
      \arrow["\Delta", from=1-4, to=2-4]
      \arrow["{\nu}"', from=1-3, to=2-3]
      \arrow["{\eta \bullet \eta}"', from=2-3, to=2-4]
      \arrow["{\iota}", from=1-5, to=2-6]
      \arrow["\eta"', from=1-5, to=2-5]
      \arrow["\varepsilon"', from=2-5, to=2-6]
    \end{tikzcd}
  \]
\end{definition}

A reconstruction result for bimonoids
in duoidal categories is proven in~\cite{booker13:tannak}.

\begin{proposition}[\cite{booker13:tannak}]\label{prop:bimonoids-give-rise-to-bimonads}
  For a monoid \(b\) in a duoidal category \duoidal{}
  there is a bijective correspondence between
  bimonoid structures on \(b\),
  and bimonad structures on the monad \(b \circ \blank\) on \(\duoidalb\).
\end{proposition}

\begin{example}\label{ex:bialgebras-gives-rise-to-bimonads}
  A bimonoid in a braided monoidal category \(\cat{C}\)
  is the same as a bimonoid in the duoidal category \(\cat{C}\) from \cref{ex:braided-cat-is-duoidal}.
  In this way one recovers the fact that an object \(b \in \cat{C}\) is a bimonoid
  if and only if the induced monad \(b \otimes \blank\) is a bimonad on \(\cat{C}\).
\end{example}

\begin{example}
  Suppose that \(k\) is a commutative ring
  and \(A\) is a commutative \(k\)-algebra.
  In~\cite[Example~6.18]{Aguiar2010} it is shown that the category of \(A\)-bimodules is duoidal,
  with
  \[
    M \bullet N \defeq M \otimes_A N \defeq \faktor{M \otimes_{k} N}{\langle ma \otimes n - m \otimes an \rangle},
  \]
  and
  \[
    M \circ N \defeq M \otimes_{A \otimes_{k} A} N \defeq \faktor{M \otimes_k N}{\langle amb \otimes n - m \otimes anb \rangle}.
  \]
  Furthermore, from~\cite[Example~6.44]{Aguiar2010} we know that a bimonoid in this duoidal category is an
  \(A\)-bialgebroid in the sense of Ravenel, see~\cite[Definition~A1.1.1]{ravenel86:compl}.
  In this setting,~\cref{prop:bimonoids-give-rise-to-bimonads} recovers a special case
  of~\cite[Theorems~5.1 and~5.4]{szlachanyi03:eilen-moore}.
\end{example}

\begin{definition}[{\cite[Section~7]{aguiar18:monad}}]\label{def:duoidal-bimonad}
  A \emph{double opmonoidal monad} on a duoidal category \(\cat{D}\) consists of
  a monad \((T, \mu, \eta)\) on \(\cat{D}\),
  together with a bimonad structures
  \((T, \bbt, \bbz)\) on \(\duoidalb\)
  and \((T, \bct, \bcz)\) on \(\duoidalw\),
  such that the following diagrams commute:
  \begin{equation}\label[diagram]{eq:pi-nu-morphisms-of-algebras}
    \begin{tikzcd}[ampersand replacement=\&]
      {T(1 \circ 1)} \& T1 \& {T\bot} \& {T(\bot \bullet \bot)} \& {T\bot} \& T1 \\
      {T1 \circ T1} \&\&\& {T\bot \bullet T\bot} \\
      {1 \circ 1} \& 1 \& \bot \& {\bot \bullet\bot} \& \bot \& 1
      \arrow["{{T \varpi}}", from=1-1, to=1-2]
      \arrow["{{T_{2,1,1}^\circ}}"', from=1-1, to=2-1]
      \arrow["{{T_0^\bullet}}", from=1-2, to=3-2]
      \arrow["{T\nu}", from=1-3, to=1-4]
      \arrow["{{T_0^\circ}}"', from=1-3, to=3-3]
      \arrow["{{T_{2,\bot,\bot}^\bullet}}", from=1-4, to=2-4]
      \arrow["{T\iota}", from=1-5, to=1-6]
      \arrow["{T^\circ_0}"', from=1-5, to=3-5]
      \arrow["{T^\bullet_0}", from=1-6, to=3-6]
      \arrow["{{T_0^\bullet \circ T_0^\bullet}}"', from=2-1, to=3-1]
      \arrow["{{T_0^\circ \bullet T_0^\circ}}", from=2-4, to=3-4]
      \arrow["\varpi"', from=3-1, to=3-2]
      \arrow["\nu", from=3-3, to=3-4]
      \arrow["\iota"', from=3-5, to=3-6]
    \end{tikzcd}
  \end{equation}
  \begin{equation}\label[diagram]{eq:cocommutative-duoidal-bimonad}
    \begin{tikzcd}[ampersand replacement=\&]
      {T((a \bullet b) \circ (c \bullet d))} \& {T((a \circ c) \bullet (b \circ d))} \\
      {T(a \bullet b) \circ T(c \bullet d)} \& {T(a \circ c) \bullet T(b \circ d)} \\
      {(Ta \bullet Tb) \circ (Tc \bullet Td)} \& {(Ta \circ Tc) \bullet (Tb \circ Td)}
      \arrow["{T \zeta_{a,b,c,d}}", from=1-1, to=1-2]
      \arrow["{T^\circ_{2,a\bullet b, c\bullet d}}"', from=1-1, to=2-1]
      \arrow["{T^\bullet_{2,a,b} \circ T^\bullet_{2,c,d}}"', from=2-1, to=3-1]
      \arrow["{T^\bullet_{2,a\circ c, b\circ d}}", from=1-2, to=2-2]
      \arrow["{T^\circ_{2,a,c} \bullet T^\circ_{2,b,d}}", from=2-2, to=3-2]
      \arrow["{\zeta_{Ta, Tb, Tc, Td}}"', from=3-1, to=3-2]
    \end{tikzcd}
  \end{equation}
\end{definition}

\begin{example}\label{ex:bimonad-in-braided-cat-is-duoidal-bimonad}
  Let \((\cat{C}, \otimes, 1)\) be a braided monoidal category with braiding \(\sigma\),
  seen as a duoidal category as in \cref{ex:braided-cat-is-duoidal}.
  A bimonad \(B\) on \((\cat{C}, \otimes, 1)\) that additionally satisfies
  the equation \(B_2 \circ B\sigma = \sigma \circ B_2\)
  is a double opmonoidal monad on \((\cat{C}, \otimes, 1, \otimes, 1)\),
  where the two opmonoidal structures are the same,
  and the commutativity of \cref{eq:pi-nu-morphisms-of-algebras} amounts to the fact that
  the monoidal structure morphisms of \(\cat{C}\) lift to the category of \(B\)-algebras;
  see \cref{prop: Moerdijk_reconstruction}.
\end{example}

\begin{example}\label{ex:cocommutative-bimonad}
  For a bialgebra \(B\) in \((\kVect, \otimes, \Bbbk)\),
  the endofunctor
  \(B \otimes \blank\) is a double opmonoidal monad in \((\kVect, \otimes, \Bbbk, \otimes, \Bbbk)\).
  For \(X, Y, Z, W \in \kVect\)
  and \(b \in B\), \(x \in X\), \(y \in Y\), \(z \in Z\), \(w \in W\),
  \cref{eq:cocommutative-duoidal-bimonad} simplifies to
  \[
    b_{(1)} \otimes x \otimes b_{(3)} \otimes z \otimes b_{(2)} \otimes y \otimes b_{(4)} \otimes w
    \,\;=\;
    b_{(1)} \otimes x \otimes b_{(2)} \otimes z \otimes b_{(3)} \otimes y \otimes b_{(4)} \otimes w,
  \]
  which is equivalent to \(b_{(1)} \otimes b_{(2)} = b_{(2)} \otimes b_{(1)}\);
  \ie, \(B\) has to be cocommutative.
\end{example}

As in the case of R-matrices for bialgebras and bimonads,
requiring that the interchange morphism of a duoidal category \(\cat{D}\)
lifts to the category of modules is a strong condition.

\begin{proposition}[{\cite[Theorem~7.2]{aguiar18:monad}}]\label{prop:cocommutative-bimonad-lifts-duoidal-structure}
  Let \(\cat{D}\) be a duoidal category and \(T\from \cat{C}\to \cat{C}\) a monad.
  Then the structure morphisms and interchange law of\, \(\cat{D}\) lift to \(\cat{D}^T\)
  if and only if
  \(T\) is a double opmonoidal monad.

  In particular, if\, \(T\) is a double opmonoidal monad, then \(\cat{D}^T\) is a duoidal category.
\end{proposition}

\subsection{R-matrices}

\noindent Instead of the situation of \cref{prop:cocommutative-bimonad-lifts-duoidal-structure},
we are instead interested in studying which additional structure one can impose on \(T\)
such that \(\cat{D}^T\) becomes duoidal,
where the interchange morphism is instead giving by ``twisting'' that of \(\cat{D}\).
This generalises so-called \emph{R-matrices} for bialgebras and bimonads.

\begin{proposition}[{\cite[Theorem~8.5]{Bruguieres2007}}]\label{prop:R-matrix-monoidal-case}
  Let \(B\) be a bimonad on the monoidal category \(\cat{C}\).
  Then R-matrices on \(B\) are in bijective correspondence with braidings on \(\cat{C}^B\).
\end{proposition}

A crucial feature of R-matrices for bimonads%
—see~\cite[Section~8.2]{Bruguieres2007}—%
is that they can be defined on not necessarily braided monoidal categories.
Our definition of duoidal R-matrices incorporates this feature.

\begin{definition}\label{def:preduoidal}
  A category \(\cat{D}\) is called \emph{preduoidal} if it is
  equipped with two monoidal structures \((\circ, \bot)\) and \((\bullet, 1)\).

  A monad \(T\) on a preduoidal category \(\cat{D}\)
  equipped with two bimonad structures over \duoidalw{} and \duoidalb{}
  is called a \emph{separately opmonoidal monad} on \(\cat{D}\).
\end{definition}

\begin{definition}\label{def:r-matrix-preduoidal}
  Let \(\cat{D}\) be a preduoidal category and \(T\) a separately opmonoidal monad on \(\cat{D}\).
  An \emph{R-matrix} on \(T\) consists of a natural transformation
  \[
    R \defeq {\{\, R_{a,b,c,d}\from (a \bullet b) \circ (c \bullet d) \nt (Ta \bullet Tc) \circ (Tb \bullet Td) \,\}}_{a,b,c,d \in \cat{D}},
  \]
  as well as morphisms of \(T\)-algebras
  \[
    \nu \from (\bot, T^{\circ}_0) \to (\bot, T^{\circ}_0) \bullet (\bot, T^{\circ}_0), \quad
    \varpi \from (1, T^{\bullet}_0) \circ (1, T^{\bullet}_0) \to (1, T^{\bullet}_0), \quad
    \iota \from (\bot, T^{\circ}_0) \to (1, T^{\bullet}_0),
  \]
  such that the tuple \((1, \varpi, \iota)\) is a monoid in \((\cat{D}^T, \circ, \bot)\);
  the tuple \((\bot, \nu, \iota)\) is a comonoid in \((\cat{D}^T, \bullet, 1)\);
  and the following diagrams commute for all \(a,b,c,d,x,y \in \cat{D}\):
  \begin{equation}\label[diagram]{eq:r-matrix-unitality1}
    \begin{tikzcd}[ampersand replacement=\&, cramped, column sep=small]
      {\bot \circ (a \bullet b)} \& {(\bot \bullet \bot) \circ (a \bullet b)} \& {(a \bullet b) \circ \bot} \& {(a \bullet b) \circ ( \bot \bullet \bot)} \\
      \& {(T\bot \circ Ta) \bullet (T\bot \circ Tb)} \&\& {(Ta \circ T\bot) \bullet (Tb \circ T\bot)} \\
      {a \bullet b} \& {(\bot \circ a) \bullet (\bot \circ b)} \& {a \bullet b} \& {(a \circ \bot) \bullet (b \circ \bot)}
      \arrow["{{\nu \circ \id}}", from=1-1, to=1-2]
      \arrow["\lambda"', from=1-1, to=3-1]
      \arrow["R", from=1-2, to=2-2]
      \arrow["{{\id \circ \nu}}", from=1-3, to=1-4]
      \arrow["\lambda"', from=1-3, to=3-3]
      \arrow["R", from=1-4, to=2-4]
      \arrow["{(T^\circ_0\circ\alpha)\bullet(T^\circ_0\circ\beta)}", from=2-2, to=3-2]
      \arrow["{(\alpha\circ T^\circ_0)\bullet(\beta\circ T^\circ_0)}", from=2-4, to=3-4]
      \arrow["{{\lambda^{-1} \bullet \lambda^{-1}}}"', from=3-1, to=3-2]
      \arrow["{{\lambda^{-1} \bullet \lambda^{-1}}}"', from=3-3, to=3-4]
    \end{tikzcd}
  \end{equation}
  \begin{equation}\label[diagram]{eq:r-matrix-unitality2}
    \scalebox{0.95}{\begin{tikzcd}[ampersand replacement=\&, cramped, column sep=small]
        {(1 \bullet a) \circ (1 \bullet b)} \& {(T1 \circ T1) \bullet (Ta \circ Tb)} \& {(a \bullet  1) \circ (b \bullet 1)} \& {(Ta \circ Tb) \bullet (T1 \circ T1)} \\
        \& {(1 \circ 1) \bullet (a \circ b)} \&\& {(a \circ b) \bullet (1 \circ 1)} \\
        {a \circ b} \& {1 \bullet (a \circ b)} \& {a \circ b} \& {(a \circ b) \bullet 1}
        \arrow["R", from=1-1, to=1-2]
        \arrow["{{\lambda \circ \lambda}}"', from=1-1, to=3-1]
        \arrow["{(T^\bullet_0\circ T^\bullet_0)\bullet(\alpha\circ\beta)}", from=1-2, to=2-2]
        \arrow["R", from=1-3, to=1-4]
        \arrow["{{\lambda \circ \lambda}}"', from=1-3, to=3-3]
        \arrow["{(\alpha\circ\beta)\bullet(T^\bullet_0\circ T^\bullet_0)}", from=1-4, to=2-4]
        \arrow["{{\varpi \bullet \id}}", from=2-2, to=3-2]
        \arrow["{{\id \bullet \varpi}}", from=2-4, to=3-4]
        \arrow["\lambda"', from=3-1, to=3-2]
        \arrow["\lambda"', from=3-3, to=3-4]
      \end{tikzcd}}
  \end{equation}
  \begin{equation}\label[diagram]{eq:r-matrix-lift}
    \scalebox{0.9}{\begin{tikzcd}[ampersand replacement=\&]
        {T((a \bullet b) \circ (c \bullet d))} \& {T((Ta \circ Tc) \bullet (Tb \circ Td))} \& {T(Ta \circ Tc) \bullet T(Tb \circ Td)} \\
        {T(a \bullet b) \circ T(c \bullet d)} \&\& {(T^2a \circ T^2c) \bullet (T^2b \circ T^2d)} \\
        {(Ta \bullet Tb) \circ (Tc \bullet Td)} \& {(T^2a \circ T^2c) \bullet (T^2b \circ T^2d)} \& {(Ta \circ Tc) \bullet (Tb \circ Td)}
        \arrow["{TR_{a,b,c,d}}", from=1-1, to=1-2]
        \arrow["{T^{\bullet}_{2,Ta \circ Tc, Tb \circ Td}}", from=1-2, to=1-3]
        \arrow["{T^{\circ}_{2,Ta,Tc} \bullet T^{\circ}_{2,Tb,Td}}", from=1-3, to=2-3]
        \arrow["{(\mu_a \circ \mu_c) \bullet (\mu_b \circ \mu_d)}", from=2-3, to=3-3]
        \arrow["{T^{\circ}_{2,a \bullet b, c \bullet d}}"', from=1-1, to=2-1]
        \arrow["{T^{\bullet}_{2,a,b} \circ T^{\bullet}_{2,c, d}}"', from=2-1, to=3-1]
        \arrow["{R_{Ta,Tb,Tc,Td}}"', from=3-1, to=3-2]
        \arrow["{(\mu_a \circ \mu_c) \bullet (\mu_b \circ \mu_d)}"', from=3-2, to=3-3]
      \end{tikzcd}}
  \end{equation}
  \begin{equation}\label[diagram]{eq:r-matrix-1}
    \scalebox{0.87}{\begin{tikzcd}[ampersand replacement=\&]
        {(a \bullet b) \circ (c \bullet d) \circ (x \bullet y)} \&\& {(a \bullet b) \circ ((Tc \circ Tx) \bullet (Td \circ Ty))} \\
        {((Ta \circ Tc) \bullet (T b \circ Td)) \circ (x \bullet y)} \&\& {(Ta \circ T(Tc \circ Tx)) \bullet (Tb \circ T(Td \circ Ty))} \\
        {(T(Ta \circ Tc) \circ Tx) \bullet (T(Tb \circ Td) \circ Ty)} \&\& {(Ta \circ (T^2c \circ T^2x)) \bullet (Tb \circ (T^2d \circ T^2y))} \\
        {((T^2a \circ T^2c) \circ Tx) \bullet ((T^2b \circ T^2d) \circ Ty)} \\
        {((Ta \circ Tc) \circ Tx) \bullet ((Tb \circ Td) \circ Ty)} \&\& {(Ta \circ (Tc \circ Tx)) \bullet (Tb \circ (Td \circ Ty))}
        \arrow["{{\mathrm{id} \circ R_{c,d,x,y}}}", from=1-1, to=1-3]
        \arrow["{{R_{a,b,c,d} \circ \mathrm{id}}}"', from=1-1, to=2-1]
        \arrow["{{R_{a,b,Tc\circ Tx, Td\circ Ty}}}", from=1-3, to=2-3]
        \arrow["{{(Ta \circ T^\circ_{2,Tc,Tx}) \bullet (Tb \circ T^\circ_{2,Td,Ty})}}", from=2-3, to=3-3]
        \arrow["{{(T^\circ_{2,Ta,Tc} \circ Tx)\bullet(T^\circ_{2,Tb,Td} \circ Ty)}}"', from=3-1, to=4-1]
        \arrow["{{R_{Ta\circ Tc, Tb\circ Td, x, y}}}"', from=2-1, to=3-1]
        \arrow["{{(Ta \circ \mu_c \circ \mu_x) \bullet (Tb \circ \mu_d \circ \mu_y)}}", from=3-3, to=5-3]
        \arrow["{{(\mu_a \circ \mu_c \circ Tx)\bullet(\mu_b \circ \mu_d \circ Ty)}}"', from=4-1, to=5-1]
        \arrow["\cong"', from=5-1, to=5-3]
      \end{tikzcd}}
  \end{equation}
  \begin{equation}\label[diagram]{eq:r-matrix-2}
    \scalebox{0.9}{\begin{tikzcd}[ampersand replacement=\&]
        {((x\bullet a)\bullet c)\circ ((y\bullet b)\bullet d)} \& {(x\bullet (a\bullet c))\circ (y\bullet (b\bullet d))} \\
        {(T(x\bullet a)\circ T(y\bullet b))\bullet (Tc\circ Td)} \& {(Tx\circ Ty)\bullet (T(a\bullet c)\circ T(b\bullet d))} \\
        {((Tx\bullet Ta)\circ (Ty\bullet Tb))\bullet (Tc\circ Td)} \& {(Tx\circ Ty)\bullet ((Ta\bullet Tc)\circ (Tb\bullet Td))} \\
        {((T^2x\circ T^2y)\bullet (T^2a\circ T^2b))\bullet (Tc\circ Td)} \& {(Tx\circ Ty)\bullet ((T^2a\circ T^2b)\bullet (T^2c\bullet T^2d))} \\
        {((Tx\circ Ty)\bullet (Ta\circ Tb))\bullet (Tc\circ Td)} \& {(Tx\circ Ty)\bullet ((Ta\circ Tb)\bullet (Tc\bullet Td))}
        \arrow["{R_{x\bullet a,c,y\bullet b,d}}"', from=1-1, to=2-1]
        \arrow["{T^\bullet _{2,x,a}\circ T^\bullet _{2,y,b}\bullet \mathrm{id}}"', from=2-1, to=3-1]
        \arrow["{R_{Tx,Ta,Ty,Tb}\bullet \mathrm{id}}"', from=3-1, to=4-1]
        \arrow[from=4-1, to=5-1]
        \arrow["{(\mu_x\circ \mu_y)\bullet (\mu_a\circ \mu_b)\bullet  \mathrm{id}}"', from=4-1, to=5-1]
        \arrow["\cong", from=1-1, to=1-2]
        \arrow["{R_{x,a\bullet c,y,b\bullet d}}", from=1-2, to=2-2]
        \arrow["{\mathrm{id}\bullet (T^\bullet _{2,a,c}\circ T^\bullet _{2,b,d})}", from=2-2, to=3-2]
        \arrow["{\mathrm{id}\bullet R_{Ta,Tc,Tb,Td}}", from=3-2, to=4-2]
        \arrow["{\mathrm{id}\bullet ((\mu_a\circ \mu_b)\bullet (\mu_c\bullet \mu_d))}", from=4-2, to=5-2]
        \arrow["\cong"', from=5-1, to=5-2]
      \end{tikzcd}}
  \end{equation}

  If \(T\) is equipped with an R-matrix, we say it is \emph{quasitriangular}.
\end{definition}

\begin{example}\label{ex:traditional-r-matrix-to-duoidal-r-matrix}
  Let \((\cat{C}, \otimes, 1)\) be a strict monoidal category,
  and \(B\) a bimonad on \(\cat{C}\).
  Suppose that \(R\) is an R-matrix on \(B\) in the sense of~\cite[Section~8.2]{Bruguieres2007},
  and let
  \[
    S \defeq {\{\, \eta_a \otimes R_{b,c} \otimes \eta_{d} \from a \otimes b \otimes c \otimes d \to Ba \otimes Bc \otimes Bb \otimes Bd \,\}}_{a,b,c,d \in \cat{C}}.
  \]
  Then \(S\), together with \(\nu\), \(\varpi\), and \(\iota\) being the identity,
  is an R-matrix on \(B\), seen as a separately opmonoidal monad on the preduoidal category \(\cat{C}\).

  For example, \cref{eq:r-matrix-unitality1,eq:r-matrix-unitality2}
  commute because \((\beta \otimes \alpha)R_{a,b}\) is a braiding by~\cite[Theorem~8.5]{Bruguieres2007}.
  \Cref{eq:r-matrix-lift} follows by \cref{fig:traditional-r-matrix-to-duoidal-r-matrix:r-matrix-lift},
  where
  \[
    B_{3; x, y, z} \from B(x \otimes y \otimes z) \to Bx \otimes By \otimes Bz
  \]
  denotes the unique natural transformation one obtains by coassociativity of \(B_2\).
  The other diagrams follow in a similar fashion.
  \begin{figure}[htbp]
    \[
      \hspace{-\the\marginparwidth+5em}
      \mathscale{0.85}{
        \begin{tikzcd}[ampersand replacement=\&,cramped]
          {B(a \otimes b \otimes c \otimes d)} \&\&\& {B(Ba \otimes Bc \otimes Bb \otimes Bd)} \\
          {B(a \otimes b) \otimes B(c \otimes d)} \& {Ba\otimes B(b\otimes c)\otimes Bd} \& {B^2a\otimes B(Bc\otimes Bb)\otimes B^2d} \& {B(Ba \otimes Bc) \otimes B(Bb \otimes Bd)} \\
          {Ba \otimes Bb \otimes Bc \otimes Bd} \& {Ba\otimes B(Bc\otimes Bb)\otimes Bd} \& {Ba\otimes B^2c\otimes B^2b\otimes Bd} \& {B^2a \otimes B^2c \otimes B^2b \otimes B^2d} \\
          \&\& {Ba \otimes B^2c \otimes B^2b \otimes Bd} \& {Ba \otimes Bc \otimes Bb \otimes Bd} \\
          {Ba \otimes Bb \otimes Bc \otimes Bd} \& {B^2a \otimes B^2c \otimes B^2b \otimes B^2d} \&\& {Ba \otimes Bc \otimes Bb \otimes Bd}
          \arrow[""{name=0, anchor=center, inner sep=0}, "{{{B(\eta_a\otimes R_{b,c}\otimes \eta_d)}}}", from=1-1, to=1-4]
          \arrow[""{name=1, anchor=center, inner sep=0}, "{{{B_{2,a \otimes b, c \otimes d}}}}"', from=1-1, to=2-1]
          \arrow["{{B_{3,a, b\otimes c, d}}}", from=1-1, to=2-2]
          \arrow["{{B_{3,Ba,Bc\otimes Bb,Bd}}}"', from=1-4, to=2-3]
          \arrow[""{name=2, anchor=center, inner sep=0}, "{{{B_{2,Ba \otimes Bc, Bb \otimes Bd}}}}", from=1-4, to=2-4]
          \arrow["{{{B_{2,a,b} \otimes B_{2,c, d}}}}"', from=2-1, to=3-1]
          \arrow[""{name=3, anchor=center, inner sep=0}, "{{B\eta_a\otimes BR_{b,c}\otimes B\eta_d}}", from=2-2, to=2-3]
          \arrow["{{\id\otimes B_{2,b,c}\otimes \id}}"{description}, from=2-2, to=3-1]
          \arrow["{{\id\otimes BR_{b,c}\otimes \id}}"', from=2-2, to=3-2]
          \arrow[""{name=4, anchor=center, inner sep=0}, "{{\mu_a\otimes \id\otimes \mu_d}}"{description}, from=2-3, to=3-2]
          \arrow[""{name=5, anchor=center, inner sep=0}, "{{\id\otimes B_{2,Bc,Bb}\otimes \id}}"{description}, from=2-3, to=3-4]
          \arrow["{{{B_{2,Ba,Bc} \otimes B_{2,Bb,Bd}}}}", from=2-4, to=3-4]
          \arrow[Rightarrow, no head, from=3-1, to=5-1]
          \arrow["{{\id\otimes B_{2,Bc,Bb}\otimes \id}}"', from=3-2, to=3-3]
          \arrow[""{name=6, anchor=center, inner sep=0}, "{{\id\otimes \mu_c\otimes \mu_b\otimes \id}}"{description}, from=3-3, to=4-4]
          \arrow["{{\mu_a\otimes \id\otimes \mu_d}}", from=3-4, to=3-3]
          \arrow["{{{\mu_a \otimes \mu_c \otimes \mu_b \otimes \mu_d}}}", from=3-4, to=4-4]
          \arrow["{{\mathsf{monad}}}"{description}, shift left=2, draw=none, from=4-3, to=5-2]
          \arrow["{\id \otimes \mu_c \otimes \mu_b \otimes \id}", from=4-3, to=5-4]
          \arrow[Rightarrow, no head, from=4-4, to=5-4]
          \arrow[""{name=7, anchor=center, inner sep=0}, "{{\id\otimes R_{Bb,Bc}\otimes \id}}", from=5-1, to=4-3]
          \arrow["{{{\eta_{Ba}\otimes R_{Bb,Bc}\otimes \eta_{Bd}}}}"', from=5-1, to=5-2]
          \arrow["{{{\mu_a \otimes \mu_c \otimes \mu_b \otimes \mu_d}}}"', from=5-2, to=5-4]
          \arrow["{{\mathsf{coassoc}}}"{description}, draw=none, from=1, to=2-2]
          \arrow["{{\mathsf{nat\ }B_3}}"{description}, draw=none, from=0, to=3]
          \arrow["{{\mathsf{coassoc}}}"{description}, draw=none, from=2, to=2-3]
          \arrow["{{\mathsf{monad}}}"{description}, draw=none, from=3, to=3-2]
          \arrow["\equiv"{description}, draw=none, from=5, to=4]
          \arrow["{{\text{\cite[(57)]{Bruguieres2007}}}}"{description}, draw=none, from=3-2, to=7]
          \arrow["\equiv"{description}, draw=none, from=3-4, to=6]
        \end{tikzcd}
      }
    \]
    \caption{Verification that \(S\) satisfies \cref{eq:r-matrix-lift}.}%
    \label{fig:traditional-r-matrix-to-duoidal-r-matrix:r-matrix-lift}
  \end{figure}
\end{example}

By~\cite[Example~8.4]{Bruguieres2007} we also obtain that
every R-matrix on a \(\Bbbk\)-bialgebra \(B\)
yields an R-matrix on \(B \otimes \blank\) in the sense of \cref{def:r-matrix-preduoidal}.

\begin{remark}\label{rmk:star-invertability}
  Note that the converse of \cref{ex:traditional-r-matrix-to-duoidal-r-matrix} is not necessarily true.
  Let \(\cat{C}\) be a monoidal category seen as a preduoidal category,
  and assume that \(T\) is a separately opmonoidal monad on \(\cat{C}\) where the two opmonoidal structures are the same.
  Then an R-matrix on \(T\) does not necessarily yield an R-matrix in the sense of~\cite[Section~8.2]{Bruguieres2007},
  since we do not require \(R\) to be \(*\)-invertible\footnote{\,%
    A natural transformation \(R\from \otimes \nt T \otimes^{\op} T\) is called \emph{\(*\)-invertible}
    if there exists an ``inverse'' natural transformation \(R^{-1} \from \bblank \otimes \blank \nt T(\blank) \otimes T(\bblank)\),
    such that
    \[
      R^{-1} * R \defeq \blank \otimes \bblank \xrightarrow{\ R\ } T(\bblank) \otimes T(\blank) \xrightarrow{\,R^{-1}\,} TT(\blank) \otimes TT(\bblank) \xrightarrow{\,\mu \otimes \mu\,} T(\blank) \otimes T(\bblank)
    \]
    is equal to \(\eta \otimes \eta\), and similarly for \(R * R^{-1}\).%
  },
  which by~\cite[Theorem~8.5]{Bruguieres2007} corresponds bijectively to the braiding on \(\cat{C}^T\) being invertible.

  By \cref{thm:r-matrices-iff-duoidal-structure} below,
  the R-matrices of \cref{def:r-matrix-preduoidal} correspond to duoidal structures on \(\cat{C}^T\).
  Since the two tensor products on \(\cat{C}^T\) agree,
  by arguments analogous to those in~\cite[Section~6.3]{Aguiar2010},
  this forces the interchange law to come from a lax braiding.
  However, there is no a priori reason for this morphism to be invertible,
  see \cref{ex:strong-duoidal-is-braided}.
\end{remark}

\subsection{From R-matrices to duoidal structures and back}

This section contains our main result,
which can be seen as an analogue of~\cite[Theorem~8.5]{Bruguieres2007},
and a non-cocommutative counterpart to~\cite[Theorem~7.2]{aguiar18:monad}.

\begin{theorem}\label{thm:r-matrices-iff-duoidal-structure}
  Let \(\cat{D}\) be a preduoidal category and suppose that
  \(T\) is a separately opmonoidal monad on \(\cat{D}\).
  For all \(T\)-algebras \((a, \alpha)\), \((b, \beta)\), \((c, \gamma)\), and \((d, \delta)\),
  a quasitriangular structure on \(T\) yields an interchange law
  \[
    \xi \defeq ((\alpha \circ \gamma) \bullet (\beta \circ \delta)) R_{a,b,c,d} \from (a \bullet b) \circ (c \bullet d) \to (a \circ c) \bullet (b \circ d)
  \]
  on \(\cat{C}^T\).
  Conversely, an interchange law \(\xi\) on \(\cat{D}^T\) gives rise to an R-matrix
  \[
    R_{a, b, c, d}\from\!
    (a \bullet b) \!\circ\! (c \bullet d)
    \xrightarrow{(\eta_a \bullet \eta_b) \circ (\eta_c \bullet \eta_d)}
    (Ta \bullet Tb) \!\circ\! (Tc \bullet Td)
    \xrightarrow{\xi_{Ta,Tb,Tc,Td}}
    (Ta \circ Tc) \!\bullet\! (Tb \circ Td)
  \]
  on \(T\).
  These constructions are mutually inverse to each other.
\end{theorem}

We split up the proof of \cref{thm:r-matrices-iff-duoidal-structure} into two individual results.

\begin{proposition}\label{prop:r-matrices-preduoidal-structure}
  Let \(\cat{D}\) be a preduoidal category and \(T\) a quasitriangular separately opmonoidal monad on \(\cat{D}\)
  with R-matrix \((R, \nu, \varpi, \iota)\).
  Then \(\cat{D}^T\) is a duoidal category,
  with structure morphisms \(\nu\), \(\varpi\), and \(\iota\),
  and interchange law
  \[
    \xi \defeq ((\alpha \circ \gamma) \bullet (\beta \circ \delta)) R_{a,b,c,d} \from (a \bullet b) \circ (c \bullet d) \to (a \circ c) \bullet (b \circ d)
  \]
  for all \((a, \alpha)\), \((b, \beta)\), \((c, \gamma)\), and \((d, \delta) \in \cat{D}^T\).
\end{proposition}
\begin{proof}
  The claim that \(\xi \in \cat{D}^T((a \bullet b) \circ (c \bullet d), (a \circ c) \bullet (b \circ d))\)
  follows from \Cref{eq:r-matrix-lift}, as seen in \cref{fig:r-matrices-duoidal-structure-1}.
  \begin{figure}[htbp]
    \[
      \begin{tikzcd}[ampersand replacement=\&,cramped]
	{T((a \bullet b) \circ (c \bullet d))} \& {T((Ta \circ Tc) \bullet (Tb \circ Td))} \& {T((a \circ c) \bullet (b \circ d))} \\
	{T(a \bullet b) \circ T(c \bullet d)} \& {T(Ta \circ Tc) \bullet T(Tb \circ Td)} \& {T(a \circ c) \bullet T(b \circ d)} \\
	\& {(T^2a \circ T^2c) \bullet (T^2b \circ T^2d)} \& {(Ta \circ Tc) \bullet (Tb \circ Td)} \\
	\& {(Ta \circ Tc) \bullet (Tb \circ Td)} \\
	{(Ta \bullet Tb) \circ (Tc \bullet Td)} \& {(T^2a \circ T^2c)\bullet(T^2b \circ T^2d)} \\
	\&\& {(a \circ c) \bullet (b \circ d)} \\
	{(a \bullet b) \circ (c \bullet d)} \& {(Ta \circ Tc) \bullet (Tb \circ Td)}
	\arrow[""{name=0, anchor=center, inner sep=0}, "{{TR_{a,b,c,d}}}", from=1-1, to=1-2]
	\arrow["{{T\xi_{a,b,c,d}}}", color=red, curve={height=-30pt}, from=1-1, to=1-3]
	\arrow["{{T^{\circ}_{2,a\bullet b,c\bullet d}}}"', from=1-1, to=2-1]
	\arrow[""{name=1, anchor=center, inner sep=0}, "{{T((\alpha \circ \gamma) \bullet (\beta \circ \delta))}}", from=1-2, to=1-3]
	\arrow["{{T^{\bullet}_{2, Ta \circ Tc, Tb \circ Td}}}"', from=1-2, to=2-2]
	\arrow["{{T^{\bullet}_{2,a \circ c, b \circ d}}}", from=1-3, to=2-3]
	\arrow["{{T^{\bullet}_{2,a,b} \circ T^{\bullet}_{2,c,d}}}"', from=2-1, to=5-1]
	\arrow["{{T^{\circ}_{2,Ta,Tc} \bullet T^{\circ}_{2,Tb,Td}}}"', from=2-2, to=3-2]
	\arrow["{{T^{\circ}_{2,a,c} \bullet T^{\circ}_{2,b,d}}}", from=2-3, to=3-3]
	\arrow[""{name=2, anchor=center, inner sep=0}, "{{(T\alpha \circ T\gamma) \bullet (T\beta \circ T\delta)}}"', from=3-2, to=3-3]
	\arrow["{{(\mu_a \circ \mu_c) \bullet (\mu_b \circ \mu_d)}}"{description}, from=3-2, to=4-2]
	\arrow["{{(\alpha \circ \gamma) \bullet (\beta \circ \delta)}}", from=3-3, to=6-3]
	\arrow[""{name=3, anchor=center, inner sep=0}, "{{(\alpha \circ \gamma)\bullet(\beta \circ \delta)}}", from=4-2, to=6-3]
	\arrow[""{name=4, anchor=center, inner sep=0}, "{{R_{Ta,Tb,Tc,Td}}}"', from=5-1, to=5-2]
	\arrow["{{(\alpha \bullet \beta)\circ(\gamma \bullet \delta)}}"', from=5-1, to=7-1]
	\arrow[""{name=5, anchor=center, inner sep=0}, "{{(\mu_a \bullet \mu_b)\circ(\mu_c \bullet\mu_d)}}", from=5-2, to=4-2]
	\arrow["{{(T\alpha\bullet T\beta)\circ(T\gamma \bullet T\delta)}}"{description}, from=5-2, to=7-2]
	\arrow["{{\xi_{a,b,c,d}}}"', color=red, curve={height=54pt}, from=7-1, to=6-3]
	\arrow[""{name=6, anchor=center, inner sep=0}, "{{R_{a,b,c,d}}}"', from=7-1, to=7-2]
	\arrow[""{name=7, anchor=center, inner sep=0}, "{{(\alpha \circ \gamma)\bullet(\beta \circ \delta)}}", from=7-2, to=6-3]
	\arrow["{\eqref{eq:r-matrix-lift}}"{description}, draw=none, from=0, to=4]
	\arrow["{\mathsf{nat}\ (T_2^\circ \bullet T_2^\circ)T_2^\bullet}"{description, pos=0.3}, shift right=3, draw=none, from=1, to=2]
	\arrow["{{\mathsf{action}}}"{description}, draw=none, from=2, to=3]
	\arrow["{\mathsf{nat}\ R}"{description}, draw=none, from=4, to=6]
	\arrow["{{\mathsf{action}}}"{description}, draw=none, from=5, to=7]
      \end{tikzcd}
    \]
    \caption{Proof that \(\xi\) is a morphism of \(T\)-algebras.}%
    \label{fig:r-matrices-duoidal-structure-1}
  \end{figure}

  \Cref{eq:middle-interchange-assoc1} follows by the commutativity of \cref{fig:r-matrices-duoidal-structure-2},
  where we have left out the respective associators for readability;
  see \cref{prop:duoidal-coherence}.
  The proof of \cref{eq:middle-interchange-assoc2} is analogous.
  \begin{figure}[htbp]
    \[
      \mathscale{0.9}{%
        \hspace{-\the\marginparwidth+2em}
        \begin{tikzcd}[nodes={font=\scriptsize}, column sep=tiny,ampersand replacement=\&]
          {(a \bullet b) \circ (c \bullet d) \circ (x \bullet y)} \& {(a \bullet b) \circ ((Tc \circ Tx) \bullet (Td \circ Ty))} \& {(a \bullet b) \circ ((c \circ x) \bullet (d \circ y))} \& {(Ta \circ T(c \circ x)) \bullet (Tb \circ T(d \circ y))} \\
          {((Ta \circ Tc) \bullet (T b \circ Td)) \circ (x \bullet y)} \& {(Ta \circ T(Tc \circ Tx)) \bullet (Tb \circ T(Td \circ Ty))} \& {(Ta \circ T^2c \circ T^2x) \bullet (Tb \circ T^2d \circ T^2y)} \& {(Ta \circ Tc \circ Tx) \bullet (Tb \circ Td \circ Ty)} \\
          {((a \circ c) \bullet (b \circ d)) \circ (x \bullet y)} \& {(T(Ta \circ Tc) \circ Tx) \bullet (T(Tb \circ Td) \circ Ty)} \\
          \&\& {(Ta \circ Tc \circ Tx) \bullet (Tb \circ Td \circ Ty)} \\
          \& {(T^2a \circ T^2c \circ Tx) \bullet (T^2b \circ T^2d \circ Ty)} \\
          {(T(a \circ c) \circ Tx) \bullet (T(b \circ d) \circ Ty)} \& {(Ta \circ Tc \circ Tx) \bullet (Tb \circ Td \circ Ty)} \&\& {(a \circ c \circ x) \bullet (b \circ d \circ y)}
          \arrow["{{\mathrm{id} \circ R_{c,d,x,y}}}", from=1-1, to=1-2]
          \arrow[""{name=0, anchor=center, inner sep=0}, "{{\id \circ ((\gamma \circ \chi) \bullet (\delta \circ \omega))}}", from=1-2, to=1-3]
          \arrow["{{R_{a,b,c\circ x, d\circ y}}}", from=1-3, to=1-4]
          \arrow["{{(Ta \circ T^\circ_{2,c,x}) \bullet (Tb \circ T^\circ_{2,d,y})}}", from=1-4, to=2-4]
          \arrow[""{name=1, anchor=center, inner sep=0}, "{{(\alpha \circ \gamma \circ \chi) \bullet (\beta \circ \delta \circ \omega)}}", from=2-4, to=6-4]
          \arrow["{{R_{a,b,c,d} \circ \mathrm{id}}}"', from=1-1, to=2-1]
          \arrow["{{((\alpha \circ \gamma) \bullet (\beta \circ \delta)) \circ \id}}"', from=2-1, to=3-1]
          \arrow["{{R_{a\circ c, b \circ d, x, y}}}"', from=3-1, to=6-1]
          \arrow[""{name=2, anchor=center, inner sep=0}, "{\raisebox{-0.8em}{\((T^\circ_{2,a,c} \circ Tx) \bullet (T^\circ_{2,b,d} \circ Ty)\)}}"', from=6-1, to=6-2]
          \arrow[""{name=3, anchor=center, inner sep=0}, "{{(\alpha \circ \gamma \circ \chi) \bullet (\beta \circ \delta \circ \omega)}}"', from=6-2, to=6-4]
          \arrow["{{R_{a,b,Tc\circ Tx, Td\circ Ty}}}"', from=1-2, to=2-2]
          \arrow[""{name=4, anchor=center, inner sep=0}, "{\raisebox{-0.8em}{\((Ta \circ T^\circ_{2,Tc,Tx}) \bullet (Tb \circ T^\circ_{2,Td,Ty})\)}}"', from=2-2, to=2-3]
          \arrow[""{name=5, anchor=center, inner sep=0}, "{\raisebox{-0.8em}{\((Ta \circ T\alpha \circ T\xi) \bullet (Tb \circ T\nu \circ T\omega)\)}}"', from=2-3, to=2-4]
          \arrow["{{(T^\circ_{2,Ta,Tc} \circ Tx)\bullet(T^\circ_{2,Tb,Td} \circ Ty)}}"{description}, from=3-2, to=5-2]
          \arrow[""{name=6, anchor=center, inner sep=0}, "{{R_{Ta\circ Tc, Tb\circ Td, x, y}}}"{description}, from=2-1, to=3-2]
          \arrow["{{(T\alpha \circ T\gamma \circ Tx)\bullet(T\beta \circ T\delta \circ Ty)}}"', from=5-2, to=6-2]
          \arrow["{{(Ta \circ \mu_c \circ \mu_x) \bullet (Tb \circ \mu_d \circ \mu_y)}}"{description}, from=2-3, to=4-3]
          \arrow["{{(\alpha \circ \gamma \circ \chi) \bullet (\beta \circ \delta \circ \omega)}}"{description}, from=4-3, to=6-4]
          \arrow[""{name=7, anchor=center, inner sep=0}, "{{(\mu_a \circ \mu_c \circ Tx)\bullet(\mu_b \circ \mu_d \circ Ty)}}"{description}, from=5-2, to=4-3]
          \arrow["{{\mathsf{action}}}"{description}, draw=none, from=2-3, to=1]
          \arrow["{{\mathsf{nat}}}"{description}, draw=none, from=0, to=5]
          \arrow["{{\mathsf{action}}}"{description}, draw=none, from=7, to=3]
          \arrow["{{\mathsf{nat}}}"{description}, draw=none, from=6, to=2]
          \arrow["{\eqref{eq:r-matrix-1}}"{description, pos=0.6}, draw=none, from=4, to=7]
        \end{tikzcd}
      }
    \]
    \caption{Proof that \(\xi\) satisfies \cref{eq:middle-interchange-assoc1}.}%
    \label{fig:r-matrices-duoidal-structure-2}
  \end{figure}

  \Cref{eq:r-matrix-unitality1,eq:r-matrix-unitality2} immediately imply \cref{eq:duoidal-cat-unitality}.
\end{proof}

\begin{proposition}\label{prop:duoidal-structure-r-matrix}
  Let \(\cat{D}\) be a preduoidal category,
  \(T\) a separately opmonoidal monad on \(\cat{D}\),
  and suppose that \(\cat{D}^T\) is a duoidal category with interchange law
  \[
    \xi_{a,b,c,d}\from (a \bullet b) \circ (c \bullet d) \to (a \circ c) \bullet (b \circ d).
  \]
  Then the structure morphisms of\, \(\cat{D}^T\), together with
  \[
    R_{a, b, c, d}\from\!
    (a \bullet b) \!\circ\! (c \bullet d)
    \xrightarrow{(\eta_a \bullet \eta_b) \circ (\eta_c \bullet \eta_d)}
    (Ta \bullet Tb) \!\circ\! (Tc \bullet Td)
    \xrightarrow{\xi_{Ta,Tb,Tc,Td}}
    (Ta \circ Tc) \!\bullet\! (Tb \circ Td)
  \]
  yield an R-matrix for \(T\).
\end{proposition}
\begin{proof}
  First, let us verify that \(R\) satisfies \cref{eq:r-matrix-lift}.
  Let \(a,b,c,d \in \cat{D}^T\);
  then the claim follows by the commutativity of \cref{fig:duoidal-structure-r-matrix:r-matrix-lift}.
  \begin{figure}[htbp]
    \[
      \mathscale{0.93}{\begin{tikzcd}[ampersand replacement=\&]
          {T((a \bullet b) \circ (c \bullet d))} \& {T((Ta \bullet Tb) \circ (Tc \bullet Td))} \& {T((Ta\circ Tc)\bullet(Tb\circ Td))} \\
          {T(a\bullet b)\circ T(c\bullet d)} \& {T(Ta\bullet Tb)\circ T(Tc\bullet Td)} \& {T(Ta\circ Tc)\bullet T(Tb\circ Td)} \\
          {(Ta \bullet Tb)\circ(Tc \bullet Td)} \& {(T^2a\bullet T^2b)\circ(T^2c\bullet T^2d)} \& {(T^2a\circ T^2c) \bullet (T^2b \circ T^2d)} \\
          \& {(Ta\bullet Tb)\circ(Tc\bullet Td)} \\
          \& {(Ta\circ Tc) \bullet (Tb \circ Td)} \\
          {(T^2a\bullet T^2b)\circ(T^2c\bullet T^2d)} \& {(T^2a\circ T^2c) \bullet (T^2b \circ T^2d)} \& {(Ta\circ Tc) \bullet (Tb \circ Td)}
          \arrow[""{name=0, anchor=center, inner sep=0}, "{\raisebox{0.3em}{\(T((\eta_{a}\bullet\eta_{b})\circ(\eta_{c}\bullet\eta_{d}))\)}}", from=1-1, to=1-2]
          \arrow["{{{T^\circ_{2;a\bullet b,c\bullet d}}}}"', from=1-1, to=2-1]
          \arrow["{\raisebox{0.3em}{\(T\xi_{Ta,Tb,Tc,Td}\)}}", from=1-2, to=1-3]
          \arrow["{{{T^\circ_{2;Ta\bullet Tb, Tc\bullet Td}}}}", from=1-2, to=2-2]
          \arrow["{{{T^\bullet_{2;Ta\circ Tc, Tb\circ Td}}}}", from=1-3, to=2-3]
          \arrow["{{{T^\bullet_{2;a,b}\circ T^\bullet_{2;c,d}}}}"', from=2-1, to=3-1]
          \arrow["{{{T^\bullet_{2;Ta,Tb}\circ T^\bullet_{2;Tc,Td}}}}", from=2-2, to=3-2]
          \arrow["{{{T^\circ_{2;Ta,Tc} \bullet T^\circ_{2;Tb,Td}}}}", from=2-3, to=3-3]
          \arrow[""{name=1, anchor=center, inner sep=0}, "{\raisebox{0.5em}{\((T\eta_{a}\bullet T\eta_{b})\circ(T\eta_{c}\bullet T\eta_{d})\)}}", from=3-1, to=3-2]
          \arrow[""{name=2, anchor=center, inner sep=0}, Rightarrow, no head, from=3-1, to=4-2]
          \arrow[""{name=3, anchor=center, inner sep=0}, "{{{(\eta_{Ta}\bullet\eta_{Tb})\circ(\eta_{Tc}\bullet\eta_{Td})}}}"', from=3-1, to=6-1]
          \arrow["{{(\mu_{a}\bullet\mu_{b})\circ(\mu_{c}\bullet\mu_{d})}}", from=3-2, to=4-2]
          \arrow[""{name=4, anchor=center, inner sep=0}, "{{{(\mu_{a}\circ\mu_{c})\bullet(\mu_{b}\circ\mu_{d})}}}", from=3-3, to=6-3]
          \arrow[""{name=5, anchor=center, inner sep=0}, "{{{\xi_{Ta,Tb,Tc,Td}}}}"', from=4-2, to=5-2]
          \arrow["{{{(\eta_{Ta}\circ\eta_{Tc})\bullet(\eta_{Tb}\circ\eta_{Td})}}}"', from=5-2, to=6-2]
          \arrow[""{name=6, anchor=center, inner sep=0}, Rightarrow, no head, from=5-2, to=6-3]
          \arrow["{\raisebox{-0.8em}{\(\xi_{T^2a,T^2b,T^2c,T^2d}\)}}"', from=6-1, to=6-2]
          \arrow["{\raisebox{-0.8em}{\((\mu_{a}\circ\mu_{c})\bullet(\mu_{b}\circ\mu_{d})\)}}"', from=6-2, to=6-3]
          \arrow["{\mathsf{nat}\ (T^\bullet_2\circ T^\bullet_2)T^\circ_2}"{description, pos=0.3}, draw=none, from=0, to=1]
          \arrow["{{\mathsf{nat\ }\xi}}"', draw=none, from=3, to=5]
          \arrow["{{T\mathsf{\ monad}}}"{description}, draw=none, from=3-2, to=2]
          \arrow["{{\xi\mathsf{\ morphism\ of\ (free)\ }T\text{-}\mathsf{algebras}}}"{description}, draw=none, from=5, to=4]
          \arrow["{{T\mathsf{\ monad}}}"{description}, draw=none, from=6, to=6-2]
        \end{tikzcd}}
    \]
    \caption{The map \(R\) satisfies \cref{eq:r-matrix-lift}.}%
    \label{fig:duoidal-structure-r-matrix:r-matrix-lift}
  \end{figure}

  The fact that \cref{eq:r-matrix-1} holds is due to \cref{fig:duoidal-structure-r-matrix-2},
  and \cref{eq:r-matrix-2} is similar.
  \begin{amssidewaysfigure}
    \vspace{1cm}
    \[
      \mathscale{0.9}{\hspace{-2cm}\begin{tikzcd}[ampersand replacement=\&]
          {(a \bullet b) \circ (c \bullet d) \circ (x \bullet y)} \& {(a\bullet b)\circ(Tc\bullet Td)\circ(Tx\bullet Ty)} \& {(a \bullet b) \circ ((Tc \circ Tx) \bullet (Td \circ Ty))} \& {(Ta\bullet Tb)\circ(T(Tc \circ Tx)\bullet T(Td\circ Ty))} \\
          {(Ta\bullet Tb)\circ (Tc \bullet Td)\circ (x\bullet y)} \\
          {((Ta \circ Tc) \bullet (T b \circ Td)) \circ (x \bullet y)} \& {(Ta\bullet Tb)\circ(Tc\bullet Td)\circ(Tx\bullet Ty)} \& {(Ta\bullet Tb)\circ((Tc \circ Tx)\bullet (Td\circ Ty))} \\
          \& {((Ta \circ Tc) \bullet (T b \circ Td)) \circ (Tx \bullet Ty)} \&\& {(Ta \circ T(Tc \circ Tx)) \bullet (Tb \circ T(Td \circ Ty))} \\
          {(T(Ta \circ Tc) \bullet T(T b \circ Td)) \circ (Tx \bullet Ty)} \\
          {(T(Ta \circ Tc) \circ Tx) \bullet (T(Tb \circ Td) \circ Ty)} \& {(Ta \circ Tc \circ Tx) \bullet (Tb \circ Td \circ Ty)} \& {(Ta \circ Tc \circ Tx) \bullet (Tb \circ Td \circ Ty)} \& {(Ta \circ T^2c \circ T^2x) \bullet (Tb \circ T^2d \circ T^2y)} \\
          \\
          {(T^2a \circ T^2c \circ Tx) \bullet (T^2b \circ T^2d \circ Ty)} \&\&\& {(Ta \circ Tc \circ Tx) \bullet (Tb \circ Td \circ Ty)}
          \arrow[""{name=0, anchor=center, inner sep=0}, "{\mathrm{id}\circ(\eta_c\bullet\eta_d)\circ(\eta_x\bullet\eta_y)}", from=1-1, to=1-2]
          \arrow["{(\eta_a\bullet\eta_b)\circ(\eta_c\bullet\eta_d)\circ\mathrm{id}}"', from=1-1, to=2-1]
          \arrow[""{name=1, anchor=center, inner sep=0}, "{\mathrm{id}\circ\xi_{Tc,Td,Tx,Ty}}", from=1-2, to=1-3]
          \arrow["{(\eta_a\bullet\eta_b)\circ\mathrm{id}}", from=1-2, to=3-2]
          \arrow[""{name=2, anchor=center, inner sep=0}, "{\raisebox{0.5em}{\((\eta_a\bullet\eta_b)\circ(\eta_{Tc\circ Tx} \bullet \eta_{Td \circ Ty})\)}}", from=1-3, to=1-4]
          \arrow["{(\eta_a\bullet\eta_b)\circ\mathrm{id}}"', from=1-3, to=3-3]
          \arrow["{\xi_{Ta,Tb,T(Tc\circ Tx), T(Td\circ Ty)}}", from=1-4, to=4-4]
          \arrow["{\xi_{Ta,Tb,Tc,Td}\circ \mathrm{id}}"', from=2-1, to=3-1]
          \arrow[""{name=3, anchor=center, inner sep=0}, "{\mathrm{id}\circ(\eta_x\bullet\eta_y)}", from=2-1, to=3-2]
          \arrow[""{name=4, anchor=center, inner sep=0}, "{\mathrm{id}\circ(\eta_x\bullet\eta_y)}"', from=3-1, to=4-2]
          \arrow["{(\eta_{Ta \circ Tc} \bullet \eta_{T b \circ Td}) \circ (\eta_x \bullet \eta_y)}"', from=3-1, to=5-1]
          \arrow[""{name=5, anchor=center, inner sep=0}, "{\mathrm{id}\circ\xi_{Tc,Td,Tx,Ty}}"', from=3-2, to=3-3]
          \arrow["{\xi_{Ta,Tb,Tc,Td}\circ \mathrm{id}}", from=3-2, to=4-2]
          \arrow["{\mathrm{id}\circ(\eta_{Tc\circ Tx} \bullet \eta_{Td \circ Ty})}"', from=3-3, to=1-4]
          \arrow["\eqref{eq:middle-interchange-assoc1}"{description}, draw=none, from=3-3, to=6-2]
          \arrow["{\xi_{Ta,Tb,Tc\circ Tx, Td\circ Ty}}"', from=3-3, to=6-3]
          \arrow["{(\eta_{Ta \circ Tc} \bullet \eta_{T b \circ Td}) \circ \mathrm{id}}", from=4-2, to=5-1]
          \arrow["{\xi_{Ta \circ Tc, Tb \circ Td, Tx, Ty}}", from=4-2, to=6-2]
          \arrow["{{{(Ta \circ T^\circ_{2,Tc,Tx}) \bullet (Tb \circ T^\circ_{2,Td,Ty})}}}", from=4-4, to=6-4]
          \arrow["{\xi_{T(Ta \circ Tc), T(T b \circ Td), Tx, Ty}}"', from=5-1, to=6-1]
          \arrow[""{name=6, anchor=center, inner sep=0}, "{{{(T^\circ_{2,Ta,Tc} \circ Tx)\bullet(T^\circ_{2,Tb,Td} \circ Ty)}}}"', from=6-1, to=8-1]
          \arrow[""{name=7, anchor=center, inner sep=0}, "{\raisebox{-0.8em}{\((\eta_{Ta \circ Tc} \circ \mathrm{id}) \bullet (\eta_{T b \circ Td} \circ \mathrm{id})\)}}", from=6-2, to=6-1]
          \arrow[""{name=8, anchor=center, inner sep=0}, "{(\mathrm{id}\circ\eta_{Tc\circ Tx}) \bullet (\mathrm{id}\circ\eta_{Td \circ Ty})}", from=6-3, to=4-4]
          \arrow[Rightarrow, no head, from=6-3, to=6-2]
          \arrow[""{name=9, anchor=center, inner sep=0}, "{\raisebox{0.5em}{\((\mathrm{id}\circ\eta_{Tc}\circ\eta_{Tx})\bullet(\mathrm{id}\circ\eta_{Td}\circ\eta_{Ty})\)}}", from=6-3, to=6-4]
          \arrow[""{name=10, anchor=center, inner sep=0}, "{(\eta_{Ta}\circ\eta_{Tc}\circ \mathrm{id})\bullet(\eta_{Tb}\circ\eta_{Td}\circ \mathrm{id})}", from=6-3, to=8-1]
          \arrow[""{name=11, anchor=center, inner sep=0}, Rightarrow, no head, from=6-3, to=8-4]
          \arrow["{{{(Ta \circ \mu_c \circ \mu_x) \bullet (Tb \circ \mu_d \circ \mu_y)}}}", from=6-4, to=8-4]
          \arrow[""{name=12, anchor=center, inner sep=0}, "{{{(\mu_a \circ \mu_c \circ Tx)\bullet(\mu_b \circ \mu_d \circ Ty)}}}"', from=8-1, to=8-4]
          \arrow["\circ\ \mathsf{functor}"{description}, draw=none, from=0, to=3]
          \arrow["\circ\ \mathsf{functor}"{description}, draw=none, from=1, to=5]
          \arrow["\circ\ \mathsf{functor}"{description}, draw=none, from=2, to=3-3]
          \arrow["\circ\ \mathsf{functor}"{description}, draw=none, from=3, to=4]
          \arrow["\circ\ \mathsf{functor}"{description}, draw=none, from=4, to=5-1]
          \arrow["{\mathsf{nat\ }\xi}", draw=none, from=3-3, to=8]
          \arrow["{\mathsf{nat\ }\xi}"{description}, draw=none, from=4-2, to=7]
          \arrow["{T\mathsf{\ \circ\text{-}bimonad}}"{description}, draw=none, from=6, to=10]
          \arrow["{T\mathsf{\ \circ\text{-}bimonad}}"{description}, draw=none, from=9, to=4-4]
          \arrow["{T\mathsf{\ monad}}"{description}, draw=none, from=6-3, to=12]
          \arrow["{T\mathsf{\ monad}}"{description}, draw=none, from=6-4, to=11]
        \end{tikzcd}}
    \]
    \caption{The R-matrix satisfies \cref{eq:r-matrix-1}.}%
    \label{fig:duoidal-structure-r-matrix-2}
  \end{amssidewaysfigure}

  It is left to show the commutativity of \cref{eq:r-matrix-unitality1,eq:r-matrix-unitality2}.
  For example, the first diagram in the former follows from the commutativity of
  \[
    \begin{tikzcd}[ampersand replacement=\&,cramped]
      {\bot \circ (a \bullet b)} \& {(\bot \bullet \bot)\circ (a\bullet b)} \& {(T\bot \bullet T\bot)\circ(Ta\bullet Tb)} \\
      {a\bullet b} \& {(\bot \circ a)\bullet(\bot \circ b)} \\
      {(\bot \circ a)\bullet(\bot \circ b)} \&\& {(T\bot \circ Ta)\bullet(T\bot \circ Tb)}
      \arrow[""{name=0, anchor=center, inner sep=0}, "{\nu \circ \mathrm{id}}", from=1-1, to=1-2]
      \arrow["\lambda"', from=1-1, to=2-1]
      \arrow[""{name=1, anchor=center, inner sep=0}, "{\raisebox{0.5em}{\((\eta_\bot\bullet\eta_\bot)\circ(\eta_a\bullet\eta_b)\)}}", from=1-2, to=1-3]
      \arrow["{\xi_{\bot,\bot,a,b}}"', from=1-2, to=2-2]
      \arrow["{\xi_{T\bot,T\bot,Ta,Tb}}", from=1-3, to=3-3]
      \arrow["{\lambda^{-1}\bullet\lambda^{-1}}"', from=2-1, to=3-1]
      \arrow[""{name=2, anchor=center, inner sep=0}, Rightarrow, no head, from=2-2, to=3-1]
      \arrow[""{name=3, anchor=center, inner sep=0}, "{(\eta_\bot\circ\eta_a)\bullet(\eta_\bot\circ\eta_b)}"{description}, from=2-2, to=3-3]
      \arrow[""{name=4, anchor=center, inner sep=0}, "{(T_0^\circ \circ \alpha)\bullet(T^\circ_0 \circ \beta)}", from=3-3, to=3-1]
      \arrow["\eqref{eq:duoidal-cat-unitality}"{description}, draw=none, from=0, to=2]
      \arrow["{\mathsf{nat\ }\xi}"{description}, draw=none, from=1, to=3]
      \arrow["{T\ \circ\text{-}\mathsf{bimonad}}"{description}, draw=none, from=2-2, to=4]
    \end{tikzcd}
  \]
  and the other diagrams are similar.
\end{proof}

\begin{proof}[Proof of \cref{thm:r-matrices-iff-duoidal-structure}]
  Combining \cref{prop:r-matrices-preduoidal-structure,prop:duoidal-structure-r-matrix},
  it is left to prove that the constructions are mutually inverse.
  \begin{align*}
    (&(\alpha \circ \gamma) \bullet (\beta \circ \delta)) \xi_{Ta,Tb,Tc,Td} ((\eta_a \bullet \eta_b) \circ (\eta_c \bullet \eta_d)) \\
     &= ((\alpha \circ \gamma) \bullet (\beta \circ \delta)) ((\eta_a \circ \eta_c) \bullet (\eta_b \circ \eta_d)) \xi_{a,b,c,d} &\text{by naturality of } \xi \\
     &= ((\alpha\eta_a \circ \gamma\eta_c) \bullet (\beta\eta_b \circ \delta\eta_d)) \xi_{a,b,c,d} &\text{by functoriality of } \bullet \text{ and } \circ \\
     &= \xi_{a,b,c,d} &\text{by monadicity of } T; \\\\
     (&(\mu_a \circ \mu_c) \bullet (\mu_b \circ \mu_d)) R_{Ta,Tb,Tc,Td} ((\eta_a \bullet \eta_b) \circ (\eta_c \bullet \eta_d)) \\
     &= ((\mu_a \circ \mu_c) \bullet (\mu_b \circ \mu_d)) ((\eta_{Ta} \circ \eta_{Tc}) \bullet (\eta_{Tb} \circ \eta_{Td})) R_{a,b,c,d} \\
     &= R_{a,b,c,d}.&&\qedhere
  \end{align*}
\end{proof}

\section{Linearly distributive monads}\label{sec:lin-dist-bimonads}

\noindent Normal duoidal categories,
see \cref{def:normal-duoidal-category},
have connections to linear logic:
in~\cite[7]{garner16:commut} it is shown that
every normal duoidal category \(\cat{D}\)
has the structure of a \emph{linearly distributive category};
see~\cite{cockett97:weakl}.
In that case, the linear distributors are given by
\begin{equation}\label{eq:normal-duoidal-to-linear-dist}
  \begin{aligned}
    \partial^\ell_\ell \from a \circ (b \bullet c) &\cong (a \bullet 1) \circ (b \bullet c) \xrightarrow{\;\; \zeta \;\;} (a \circ b) \bullet (1 \bullet c) \cong (a \circ b) \bullet c, \\
    \partial^\ell_r \from a \circ (b \bullet c) &\cong (1 \bullet a) \circ (b \bullet c) \xrightarrow{\;\; \zeta \;\;} (1 \circ b) \bullet (a \circ c) \cong b \bullet (a \circ c), \\
    \partial^r_{\ell} \from (b \bullet c) \circ a &\cong (b \bullet c) \circ (a \bullet 1) \xrightarrow{\;\; \zeta \;\;} (b \circ a) \bullet (c \circ 1) \cong (b \circ a) \bullet c. \\
    \partial^r_r \from (b \bullet c) \circ a &\cong (b \bullet c) \circ (1 \bullet a) \xrightarrow{\;\; \zeta \;\;} (b \circ 1) \bullet (c \circ a) \cong b \bullet (c \circ a).
  \end{aligned}
\end{equation}

Since by~\cite[Theorem~5.18]{malkiewich22:coher},
normal duoidal categories satisfy a much stronger form of coherence,
structures on them require fewer axioms to be fully specified.
For example,
if \(T\) is a double opmonoidal monad on a normal duoidal category \(\cat{D}\),
then the following diagram commutes:
\begin{equation}\label{eq:normal-B0-conjugate}
  \begin{tikzcd}[ampersand replacement=\&,cramped,row sep=small]
    1 \&\& T1 \&\& 1 \\
    \\
    \bot \&\& {T\bot} \&\& \bot
    \arrow["{{\eta_1}}"', from=1-1, to=1-3]
    \arrow[""{name=0, anchor=center, inner sep=0}, curve={height=-28pt}, equals, from=1-1, to=1-5]
    \arrow[""{name=1, anchor=center, inner sep=0}, "\cong"', from=1-1, to=3-1]
    \arrow["{{T^\bullet_0}}"', from=1-3, to=1-5]
    \arrow[""{name=2, anchor=center, inner sep=0}, "{T(\cong)}", from=1-3, to=3-3]
    \arrow[""{name=3, anchor=center, inner sep=0}, "\cong", from=1-5, to=3-5]
    \arrow["{{\eta_\bot}}", from=3-1, to=3-3]
    \arrow[""{name=4, anchor=center, inner sep=0}, curve={height=28pt}, equals, from=3-1, to=3-5]
    \arrow["{{T^\circ_0}}", from=3-3, to=3-5]
    \arrow["{\mathsf{bimonad}}"{description, pos=0.7}, draw=none, from=0, to=1-3]
    \arrow["{\mathsf{nat}\ \eta}"{description}, draw=none, from=1, to=2]
    \arrow["\eqref{eq:pi-nu-morphisms-of-algebras}"{description}, draw=none, from=2, to=3]
    \arrow["{\mathsf{bimonad}}"{description, pos=0.2}, draw=none, from=3-3, to=4]
  \end{tikzcd}
\end{equation}
In particular, \(T_0^{\bullet}\) and \(T_0^{\circ}\) are conjugate by isomorphisms:
\[
  (T1 \xrightarrow{\;\;T_0^{\bullet}\;\;} 1)
  g= (T1 \xrightarrow{\;T(\cong)\;} T\bot \xrightarrow{\;\;T_0^{\circ}\;\;} \bot \xrightarrow{\;\cong^{-1}\;} 1).
\]

Further, in the above setting, \cref{eq:pi-nu-morphisms-of-algebras} automatically holds.
For simplicity, assume \(\cat{D}\) to be strict, and write \(T_0 \defeq T_0^{\bullet} = T_0^{\circ}\).
Then, for example, we have
\begin{equation}\label{eq:cocomm-trialg-unit-automatic}
  \begin{tikzcd}[ampersand replacement=\&]
    {T(1 \circ 1)} \&\& T1 \\
    {T1 \circ T1} \& T1 \\
    {1 \circ 1} \&\& 1
    \arrow[""{name=0, anchor=center, inner sep=0}, "{{T \varpi}}", curve={height=-18pt}, from=1-1, to=1-3]
    \arrow["{{T_{2,1,1}^\circ}}"', from=1-1, to=2-1]
    \arrow[""{name=1, anchor=center, inner sep=0}, "\varpi"', curve={height=18pt}, from=3-1, to=3-3]
    \arrow["{{T_0}}", from=1-3, to=3-3]
    \arrow["{{T_0 \circ T_0}}"', from=2-1, to=3-1]
    \arrow[""{name=2, anchor=center, inner sep=0}, "{T(\cong)}"', from=1-1, to=1-3]
    \arrow[""{name=3, anchor=center, inner sep=0}, "\cong", from=3-1, to=3-3]
    \arrow["{T_0\circ \mathrm{id}}"', from=2-1, to=2-2]
    \arrow[Rightarrow, no head, from=1-3, to=2-2]
    \arrow["{T_0}", from=2-2, to=3-3]
    \arrow["{\mathsf{coherence}}"{description}, draw=none, from=0, to=2]
    \arrow["{\mathsf{coherence}}"{description}, draw=none, from=3, to=1]
  \end{tikzcd}
\end{equation}
The other diagram is similar.

\begin{remark}\label{rmk:non-planar-lcds}
  Sometimes, one considers only so-called \emph{non-planar} linearly distributive categories,
  see~\cite[Section~2.1]{cockett97:weakl}.
  These are categories in which only \(\partial_{\ell}^{\ell}\) and \(\partial_r^r\) of \cref{eq:normal-duoidal-to-linear-dist} exist.
  What we call a linearly distributive category is referred to as a \emph{planar} linearly distributive category in \emph{ibid}.
\end{remark}

Conditions for a comonad to lift the (non-planar) linear distributive structure of its base category to its category of coalgebras were defined in~\cite{pastro12:note}.
For the convenience of the reader, the next proposition expresses this relation in terms of monads.

\begin{proposition}[{\cite[Proposition~2.1]{pastro12:note}}]\label{prop:linearly-distributive-monad}
  Let \((\cat{L}, \otimes, \odot)\) be a non-planar linearly distributive category,
  and suppose that the monad \(T\) on \(\cat{L}\) is separately opmonoidal.
  If the diagrams
  \begin{equation}\label[diagram]{eq:linearly-distributive-monad-1}
    \begin{tikzcd}[ampersand replacement=\&]
	{T(a \otimes (b \odot c))} \& {Ta \otimes T(b \odot c)} \& {Ta \otimes (Tb \odot Tc)} \\
	{T((a \otimes b) \odot c)} \& {T(a \otimes b) \odot Tc} \& {(Ta \otimes Tb) \odot Tc}
	\arrow["{T\partial_l}"', from=1-1, to=2-1]
	\arrow["{T^\otimes_{2, a, b\odot c}}", from=1-1, to=1-2]
	\arrow["{Ta \otimes T^\odot_{2, b,c}}", from=1-2, to=1-3]
	\arrow["{T^\odot_{2, a\otimes b, c}}"', from=2-1, to=2-2]
	\arrow["{T^\otimes_{2,a,b} \odot Tc}"', from=2-2, to=2-3]
	\arrow["{\partial_l}", from=1-3, to=2-3]
      \end{tikzcd}
  \end{equation}
  \begin{equation}\label[diagram]{eq:linearly-distributive-monad-2}
      \begin{tikzcd}[ampersand replacement=\&]
	{T((b \odot c)\otimes a)} \& {T(b \odot c) \otimes Ta} \& {(Tb \odot Tc) \otimes Ta} \\
	{T(b \odot (c \otimes a))} \& {Tb \odot T(c \otimes a)} \& {Tb \odot (Tc \otimes Ta)}
	\arrow["{T\partial_r}"', from=1-1, to=2-1]
	\arrow["{T^\otimes_{2,b\odot c, a}}", from=1-1, to=1-2]
	\arrow["{T^\odot_{2,b,c}\otimes Ta}", from=1-2, to=1-3]
	\arrow["{T^\odot_{2,b,c\otimes a}}"', from=2-1, to=2-2]
	\arrow["{Tb \odot T^\otimes_{2,c,a}}"', from=2-2, to=2-3]
	\arrow["{\partial_r}", from=1-3, to=2-3]
      \end{tikzcd}
  \end{equation}
  commute for all \(T\)-algebras \(a\), \(b\), and \(c\), then \(\cat{L}^T\) is non-planar linearly distributive.
\end{proposition}

\begin{example}\label{ex:linear-dist-bimonad-in-monoidal-setting}
  Every monoidal category \(\cat{C}\) is a linearly distributive category,
  setting \(\otimes = \odot\).
  The linear distributors are the associator (and its inverse) of \(\cat{C}\).
  A bimonad \((B, B_2, B_0)\) on \((\cat{C}, \otimes)\) therefore satisfies all assumptions of \cref{prop:linearly-distributive-monad}.
  \Cref{eq:linearly-distributive-monad-1,eq:linearly-distributive-monad-2}
  reduce to the coassociativity of \(B_2\).
\end{example}

However, lifting the interchange morphism of a normal duoidal category may be more involved than lifting only the non-planar linear distributors,
much like lifting the preduoidal structure is much easier than lifting the entire duoidal structure.

\begin{example}
  Let \(\cat{C}\) be a braided monoidal category,
  which is normal duoidal by \cref{ex:braided-cat-is-duoidal}.
  As such, the linear distributor \(\partial_\ell^{\ell}\) is the isomorphism
  \[
    \partial_{\ell}^{\ell} \from x \otimes (y \otimes z)
    \cong x \otimes (1 \otimes y) \otimes z
    \xrightarrow{x \otimes \sigma_{1, y} \otimes z} x \otimes (y \otimes 1) \otimes z
    \cong x \otimes (y \otimes z),
  \]
  and \(\partial^r_r\) is similar.
  By \cref{prop:linearly-distributive-monad},
  this structure lifts to the Eilenberg–Moore category of \(B \otimes \blank\),
  which is equal to the category of \(B\)-modules on \(\cat{C}\).
  Analogously to \cref{ex:linear-dist-bimonad-in-monoidal-setting},
  \cref{eq:linearly-distributive-monad-1,eq:linearly-distributive-monad-2}
  reduce to the coassociativity of \(\Delta\).

  However, it is not true that the modules over an arbitrary bialgebra are braided monoidal;
  see for example~\cite[Example~8.3.5]{Etingof2015}.
  In other words, the planar structure
  \begin{align*}
    \partial_r^{\ell} &\from x \otimes (y \otimes z)
             \cong 1 \otimes (x \otimes y) \otimes z
             \xrightarrow{1 \otimes \sigma_{x, y} \otimes z} 1 \otimes (y \otimes x) \otimes z
             \cong y \otimes (x \otimes z), \\
    \partial_\ell^r &\from (x \otimes y) \otimes z
           \cong x \otimes (y \otimes z) \otimes 1
           \xrightarrow{x \otimes \sigma_{y, z} \otimes 1} x \otimes (z \otimes y) \otimes 1
           \cong (x \otimes z) \otimes y,
  \end{align*}
  does not lift to \(B\)-modules.
\end{example}

As stated in the introduction,
planar duoidal categories also capture and generalise the notion of a braiding,
much like duoidal categories do—as such, we shall focus on this case from now on.
We begin with a straightforward reformulation of \cref{prop:linearly-distributive-monad}.

\begin{proposition}\label{prop:planar-linearly-distributive-monad}
  Let \((\cat{L}, \otimes, \odot)\) be a linearly distributive category
  with a separately opmonoidal monad \(T\) on it.
  If, in addition to \cref{eq:linearly-distributive-monad-1,eq:linearly-distributive-monad-2},
  the following diagrams
  commute for all \(T\)-algebras \(a\), \(b\), and \(c\):xs
  \begin{equation}\label[diagram]{eq:linearly-distributive-monad-3}
    \begin{tikzcd}[ampersand replacement=\&]
      {B(a \otimes (b \odot c))} \& {Ba \otimes B(b \odot c)} \& {Ba \otimes (Bb \odot Bc)} \\
      {B(b \odot (a \otimes c))} \& {Bb \odot B(a\otimes c)} \& {Bb \odot (Ba \otimes Bc)}
      \arrow["{B^\otimes_{2; a, b\odot c}}", from=1-1, to=1-2]
      \arrow["{B\partial_r^\ell}"', from=1-1, to=2-1]
      \arrow["{Ba \otimes B^\odot_{2; b, c}}", from=1-2, to=1-3]
      \arrow["{\partial_r^\ell}", from=1-3, to=2-3]
      \arrow["{B^\odot_{2; b, a\otimes c}}"', from=2-1, to=2-2]
      \arrow["{Bb \odot B^\otimes_{2;a,c}}"', from=2-2, to=2-3]
    \end{tikzcd}
  \end{equation}
  \begin{equation}\label[diagram]{eq:linearly-distributive-monad-4}
    \begin{tikzcd}[ampersand replacement=\&]
      {B((a \odot b)\otimes c)} \& {B(a \odot b) \otimes Bc} \& {(Ba \odot Bb) \otimes Bc} \\
      {B(a \odot (c \otimes b))} \& {Ba \odot B(c \otimes b)} \& {Ba \odot (Bc \otimes Bb)}
      \arrow["{B^\otimes_{2; a\odot b, c}}", from=1-1, to=1-2]
      \arrow["{B\partial_\ell^r}"', from=1-1, to=2-1]
      \arrow["{B^\odot_{2; a, b}\otimes Bc}", from=1-2, to=1-3]
      \arrow["{\partial_r^\ell}", from=1-3, to=2-3]
      \arrow["{B^\odot_{2; a, c\otimes b}}"', from=2-1, to=2-2]
      \arrow["{Ba \odot B^\otimes_{2;c,b}}"', from=2-2, to=2-3]
    \end{tikzcd}
  \end{equation}
  then \(\cat{L}^T\) is linearly distributive.
\end{proposition}

\begin{example}
  Let \(B \in \kVect\) be a bialgebra.
  Focusing on the planar linear distributor \(\partial_{r}^{\ell}\),
  for all \(b \in B\), \(x \in a\), \(y \in b\), and \(z \in c\),
  \cref{eq:linearly-distributive-monad-3} reduces to the equality
  \[
    b_{(1)} \otimes y \otimes b_{(2)} \otimes x \otimes b_{(3)} \otimes z
    =
    b_{(2)} \otimes y \otimes b_{(1)} \otimes x \otimes b_{(3)} \otimes z,
  \]
  which is easily seen to be equivalent to \(b_{(2)} \otimes b_{(1)} = b_{(1)} \otimes b_{(2)}\).
\end{example}

Thus, linearly distributive monads seem to be connected to the double opmonoidal monads of \cref{sec:duoidal-bimonads}.

\begin{proposition}\label{prop:cocommutative-trimonads-are-pastro-monads}
  Let \((\cat{D}, \bullet, \circ, 1)\) be a normal duoidal category.
  Then double opmonoidal monads on \(\cat{D}\)
  are linear distributive bimonads on \(\cat{D}\).
\end{proposition}
\begin{proof}
  Let \(B\) be a double opmonoidal monad on \(\cat{D}\).
  Then the left-left linear distributor \(\partial^{\ell}_\ell\) is given by
  \[
    a \circ (b \bullet c) \cong (a \bullet 1) \circ (b \bullet c) \xrightarrow{\;\; \zeta \;\;} (a \circ b) \bullet (1 \circ c) \cong (a \circ b) \bullet c.
  \]
  Now, \cref{eq:linearly-distributive-monad-1} is satisfied by the commutativity of
  \cref{fig:cocommutative-trimonads-are-pastro-monads-1};
  \cref{eq:linearly-distributive-monad-2} is similar.
  \begin{figure}[htbp]
    \[
      \mathscale{0.8}{
        \begin{tikzcd}[ampersand replacement=\&]
          {B(a \circ (b \bullet c))} \& {Ba \circ B(b \bullet c)} \&\&\& {Ba \circ (Bb \bullet Bc)} \\
          {B((a \bullet 1) \circ (b\bullet c))} \& {B(a \bullet 1) \circ B(b \bullet c)} \& {(Ba \bullet B1)\circ (Bb \bullet Bc)} \&\& {(Ba \bullet 1)\circ (Bb \bullet Bc)} \\
          {B((a \circ b) \bullet (1 \circ c))} \& {B(a \circ b) \bullet B(1 \circ c)} \& {(Ba \circ Bb) \bullet (B1 \circ Bc)} \&\& {(Ba \circ Bb) \bullet (1 \circ Bc)} \\
          {B((a \circ b) \bullet (\bot \circ c))} \\
          \& {B(a \circ b) \bullet B(\bot \circ c)} \& {(Ba \circ Bb)\bullet (B\bot \circ Bc)} \&\& {(Ba \circ Bb) \bullet (\bot \circ Bc)} \\
          \\
          {B((a \circ b) \bullet c)} \& {B(a \circ b)      \bullet Bc} \&\&\& {(Ba \circ Bb) \bullet Bc}
          \arrow["{B(\cong)}"', from=1-1, to=2-1]
          \arrow["B\zeta_{a,1,b,c}"', from=2-1, to=3-1]
          \arrow[""{name=0, anchor=center, inner sep=0}, "{B^\bullet_{2,a\circ b, c}}"', from=7-1, to=7-2]
          \arrow[""{name=1, anchor=center, inner sep=0}, "{B^\circ_{2,a,b}\bullet Bc}"', from=7-2, to=7-5]
          \arrow[""{name=2, anchor=center, inner sep=0}, "{B^\circ_{2,a,b\bullet c}}", from=1-1, to=1-2]
          \arrow[""{name=3, anchor=center, inner sep=0}, "{Ba\circ B^\bullet_{2,b,c}}", from=1-2, to=1-5]
          \arrow["\cong", from=1-5, to=2-5]
          \arrow["\zeta_{Ba,1,Bb,Bc}", from=2-5, to=3-5]
          \arrow["{\zeta_{Ba, B1, Bb, Bc}}", from=2-3, to=3-3]
          \arrow[""{name=4, anchor=center, inner sep=0}, "{B^\circ_{2,a\bullet 1, b\bullet c}}", from=2-1, to=2-2]
          \arrow["{B^\bullet_{2,a,1}\circ B^\bullet_{2,b,c}}", from=2-2, to=2-3]
          \arrow[""{name=5, anchor=center, inner sep=0}, "{B^\bullet_{2,a\circ b, 1 \circ c}}"', from=3-1, to=3-2]
          \arrow["{B^\circ_{2,a,b} \bullet B^\circ_{2,1,c}}"', from=3-2, to=3-3]
          \arrow["\eqref{eq:cocommutative-duoidal-bimonad}"{description}, draw=none, from=2-2, to=3-2]
          \arrow[""{name=6, anchor=center, inner sep=0}, "{(Ba\bullet B^\bullet_0)\circ (Bb \bullet Bc)}", from=2-3, to=2-5]
          \arrow[""{name=7, anchor=center, inner sep=0}, "{(Ba \circ Bb) \bullet (B^\bullet_0 \circ Bc)}"', from=3-3, to=3-5]
          \arrow["{B(\cong) \circ B(b\bullet c)}", from=1-2, to=2-2]
          \arrow["{B(\cong)}"', from=4-1, to=7-1]
          \arrow["{B(a \circ b) \bullet B(\cong)}", from=3-2, to=5-2]
          \arrow["{B(a \circ b) \bullet B(\cong)}", from=5-2, to=7-2]
          \arrow["{B(\cong)}"', from=3-1, to=4-1]
          \arrow["\cong", curve={height=-12pt}, from=3-5, to=5-5]
          \arrow["\cong", from=5-5, to=7-5]
          \arrow["{B^\circ_{2,a,b} \bullet B^\circ_{2,\bot,c}}"', from=5-2, to=5-3]
          \arrow["{(Ba \circ Bb)\bullet(B^\circ_0 \circ Bc)}"', from=5-3, to=5-5]
          \arrow["{\cong^{-1}}", curve={height=-12pt}, from=5-5, to=3-5]
          \arrow["\eqref{eq:cocomm-trialg-unit-automatic}"{description}, draw=none, from=3-3, to=5-3]
          \arrow["{\mathsf{nat}\ B_2^{\circ}}"{description}, draw=none, from=2, to=4]
          \arrow["{\mathsf{nat}\ B_2^{\bullet}}"{description}, draw=none, from=5, to=0]
          \arrow["{(B, B^\bullet_2, B^\bullet_0)\ \mathsf{bimonad}}"{description}, draw=none, from=3, to=2-3]
          \arrow["{\mathsf{nat}\ \zeta}"{description}, draw=none, from=6, to=7]
          \arrow["{(B, B^\circ_2, B^\circ_0)\ \mathsf{bimonad}}"{description}, draw=none, from=5-3, to=1]
        \end{tikzcd}
      }
    \]
    \caption{The left-left linear distributor satisfies \cref{eq:linearly-distributive-monad-1}.}%
    \label{fig:cocommutative-trimonads-are-pastro-monads-1}
  \end{figure}

  \Cref{eq:linearly-distributive-monad-3} is satisfied by
  \cref{fig:cocommutative-trimonads-are-pastro-monads-2}%
  ---where we have assumed the normal duoidal structure to be strict for ease of readability---%
  and \cref{eq:linearly-distributive-monad-4} follows similarly.
  \begin{figure}[htbp]
    \[
      \mathscale{0.87}{\begin{tikzcd}[ampersand replacement=\&]
          {B(a \circ (b \bullet c))} \& {Ba \circ B(b \bullet c)} \&\& {Ba \circ (Bb \bullet Bc)} \\
          {B((1 \bullet a) \circ (b \bullet c))} \& {B(1 \bullet a) \circ B(b \bullet c)} \& {(B1 \bullet Ba) \circ (Bb \bullet Bc)} \& {(1 \bullet Ba) \circ (Bb \bullet Bc)} \\
          {B((1 \circ b) \bullet (a \circ c))} \& {B(1\circ b) \bullet B(a\circ c)} \& {(B1 \circ Bb) \bullet (Ba \circ Bc)} \& {(1 \circ Bb) \bullet (Ba \circ Bc)} \\
          {B(b \bullet (a \circ c))} \& {Bb \bullet B(a\circ c)} \&\& {Bb \bullet (Ba \circ Bc)}
          \arrow["{{B^\circ_{2; a, b\bullet c}}}", from=1-1, to=1-2]
          \arrow[Rightarrow, no head, from=1-1, to=2-1]
          \arrow[""{name=0, anchor=center, inner sep=0}, "{{\mathrm{id} \circ B^\bullet_{2; b, c}}}", from=1-2, to=1-4]
          \arrow[Rightarrow, no head, from=1-2, to=2-2]
          \arrow[Rightarrow, no head, from=1-4, to=2-4]
          \arrow["{B^\circ_{2;1\bullet a,b\bullet c}}", from=2-1, to=2-2]
          \arrow["{B\xi_{1,a,b,c}}"', from=2-1, to=3-1]
          \arrow["{{B^\bullet_{2;1,a} \circ B^\bullet_{2; b, c}}}", from=2-2, to=2-3]
          \arrow["{\eqref{eq:cocommutative-duoidal-bimonad}}"{description}, draw=none, from=2-2, to=3-2]
          \arrow["{(B_0 \bullet \mathrm{id})\circ\mathrm{id}}", from=2-3, to=2-4]
          \arrow[""{name=1, anchor=center, inner sep=0}, "{\xi_{B1,Ba,Bb,Bc}}"', from=2-3, to=3-3]
          \arrow[""{name=2, anchor=center, inner sep=0}, "{\xi_{1,Ba,Bb,Bc}}", from=2-4, to=3-4]
          \arrow["{{B^\bullet_{2; 1\circ b, a\circ c}}}"', from=3-1, to=3-2]
          \arrow["{{B^\circ_{2;1,b} \bullet B^\circ_{2;a,c}}}"', from=3-2, to=3-3]
          \arrow["{B\ \mathsf{bimonad}}"{description}, draw=none, from=3-2, to=4-2]
          \arrow["{(B_0 \circ \mathrm{id})\bullet\mathrm{id}}"{description}, from=3-3, to=4-4]
          \arrow[Rightarrow, no head, from=4-1, to=3-1]
          \arrow["{{B^\bullet_{2; b, a\circ c}}}"', from=4-1, to=4-2]
          \arrow["{{\mathrm{id} \bullet B^\circ_{2;a,c}}}"', from=4-2, to=4-4]
          \arrow[Rightarrow, no head, from=4-4, to=3-4]
          \arrow["{B\ \mathsf{bimonad}}"{description}, draw=none, from=0, to=2-3]
          \arrow["{\mathsf{nat\ }\xi}"', draw=none, from=1, to=2]
        \end{tikzcd}}
    \]
    \caption{The right-left linear distributor satisfies \cref{eq:linearly-distributive-monad-3}.}%
    \label{fig:cocommutative-trimonads-are-pastro-monads-2}
  \end{figure}
\end{proof}

\bibliographystyle{alpha}
\bibliography{main}

\end{document}